\documentclass[12pt]{amsart}

\usepackage{amsxtra,amssymb,amsthm,amsmath,amscd,url}
\usepackage[latin1]{inputenc}
\usepackage{fullpage}

\renewcommand{\leq}{\leqslant}
\renewcommand{\geq}{\geqslant}

\numberwithin{equation}{section}
\def\stacksum#1#2{{\stackrel{{\scriptstyle #1}}
{{\scriptstyle #2}}}}

\newcommand{\Cc}{\mathbf{C}}

\newcommand{\Zz}{\mathbf{Z}}

\newcommand{\Rr}{\mathbf{R}}

\newcommand{\Hh}{\mathbf{H}}
\newcommand{\Qq}{\mathbf{Q}}
\newcommand{\Fp}{\mathbf{F}}
\newcommand{\Tt}{\mathbf{T}}
\newcommand{\G}{\mathbf{G}}

\newcommand{\expect}{\text{\boldmath$E$}}

\newcommand{\charfun}{\mathbf{1}}

\newcommand{\mods}[1]{\,(\mathrm{mod}\,{#1})}

\newcommand{\ra}{\rightarrow}
\newcommand{\lra}{\longrightarrow}

\newcommand{\injecte}{\hookrightarrow}

\newcommand{\fleche}[1]{\stackrel{#1}{\lra}}

\DeclareMathOperator{\spec}{Spec}

\DeclareMathOperator{\Reel}{Re}

\DeclareMathOperator{\frob}{\mathrm{Fr}}
\DeclareMathOperator{\Gal}{Gal}

\DeclareMathOperator{\Tr}{Tr}
\DeclareMathOperator{\Hom}{Hom}

\newcommand{\eps}{\varepsilon}

\newcommand{\sheaf}[1]{\mathcal{{#1}}}

\DeclareMathSymbol{\gena}{\mathord}{letters}{"3C}
\DeclareMathSymbol{\genb}{\mathord}{letters}{"3E}

\theoremstyle{plain}
\newtheorem{theorem}{Theorem}[section]

\newtheorem{lemma}[theorem]{Lemma}
\newtheorem{corollary}[theorem]{Corollary}

\newtheorem{proposition}[theorem]{Proposition}

\theoremstyle{remark}
\newtheorem{remark}[theorem]{Remark}

\theoremstyle{definition}

\newtheorem{example}[theorem]{Example}

\newcommand{\uple}[1]{\text{\boldmath${#1}$}}

\newcommand{\zeros}{\mathcal{Z}}
\newcommand{\nzeros}{\mathcal{\tilde{Z}}}
\newcommand{\vect}[1]{\langle\,{#1}\,\rangle_{a}}
\newcommand{\mult}[1]{\langle\,{#1}\,\rangle_{m}}
\newcommand{\reladd}[1]{\mathrm{Rel}({#1})_{a}}
\newcommand{\relmul}[1]{\mathrm{Rel}({#1})_{m}}
\newcommand{\relmult}[1]{\mathrm{Rel}_0({#1})_{m}}
\newcommand{\reltriv}[1]{\mathrm{Triv}({#1})_{m}}
\newcommand{\pione}[1]{\pi_1({{#1}},\eta_{{\scriptscriptstyle{{#1}}}})}
\newcommand{\pioneg}[1]{\pi_1(\bar{{{#1}}},\eta_{{\scriptscriptstyle{\bar{#1}}}})}
\newcommand{\pionet}[1]{\pi_1^t(\bar{{{#1}}},\eta_{{\scriptscriptstyle{\bar{#1}}}})}

\begin{document}

\title{The large sieve, monodromy and zeta functions of algebraic
  curves, II: independence of the zeros}
\author{E. Kowalski}


\address{ETH Zurich -- D-MATH, R\"amistrasse 101, 8092 Z\"urich,
  Switzerland
}
\email{kowalski@math.ethz.ch} 

\keywords{Families of curves over finite fields, zeta and
  $L$-functions, large sieve, Grand Simplicity Hypothesis, Random
  Matrix Models of $L$-functions, Chebychev bias, Frobenius tori}

\subjclass[2000]{Primary 11N35, 11G25, 14G15; Secondary 14D10, 11M99}

\begin{abstract}Using the sieve for Frobenius developed earlier by the
  author, we show that in a certain sense, the roots of the
  $L$-functions of most algebraic curves over finite fields do not
  satisfy any non-trivial (linear or multiplicative) dependency
  relations. This can be seen as an analogue of conjectures of
  $\Qq$-linear independence among ordinates of zeros of $L$-functions
  over number fields.  As a corollary of independent interest, we find
  for ``most'' pairs of distinct algebraic curves over a finite field
  the form of the distribution of the (suitably normalized) difference
  between the number of rational points over extensions of the ground
  field. The method of proof also emphasizes the relevance of Random
  Matrix models for this type of arithmetic questions. We also
  describe an alternate approach, suggested by N. Katz, which relies
  on Serre's theory of Frobenius tori.
\end{abstract}

\maketitle

\section{Introduction}

In a number of studies of the fine distribution of primes, there
arises the issue of the existence of linear dependence relations, with
rational coefficients, among zeros (or rather ordinates of zeros) of
the Riemann zeta function, or more generally of Dirichlet
$L$-functions.  This was important in disbelieving (then disproving,
as done by Odlyzko and te Riele) Mertens's Conjecture
\begin{equation}\label{eq-mus}
\Bigl|\sum_{n\leq x}{\mu(n)}\Bigr|<\sqrt{x},\quad\quad\text{ for } x
\geq 2
\end{equation}
as Ingham~\cite{ingham} showed how it implied that the zeros of
$\zeta(s)$ are $\Qq$-linearly dependent.\footnote{\ In fact, that
  zeros ``arbitrarily high'' on the critical line are linearly
  dependent. For the most recent work in studying the left-hand side
  of~(\ref{eq-mus}) using the assumption of linear independence, see
  the work of N. Ng~\cite{ng}.} 
\par
More recently, this turned out to be crucial in understanding the
``Chebychev bias'' in the distribution of primes in arithmetic
progressions (the apparent preponderance of primes $\equiv 3\mods{4}$
compared to those $\equiv 1\mods{4}$, and generalizations of this), as
discussed in depth by Rubinstein and
Sarnak~\cite{rubinstein-sarnak}. They introduce the ``Grand Simplicity
Conjecture'' as the statement that the set of all ordinates
$\gamma\geq 0$ of the non-trivial zeros $\rho$ of Dirichlet
$L$-functions $L(s,\chi)$ are $\Qq$-linearly independent when $\chi$
runs over primitive Dirichlet characters and the zeros are counted
with multiplicity (indeed, ``simplicity'' relates to the particular
corollary of this conjecture that all zeros of Dirichlet $L$-functions
are simple).
\par
Building on the fact that our current knowledge of the behavior of
zeros of zeta functions of (smooth, projective, geometrically
connected) algebraic curves over finite fields is somewhat more
extensive, we consider analogues of this type of independence
questions in the context of finite fields.
Let $C/\Fp_q$ be such an algebraic curve over a finite field with $q$
elements and characteristic $p$, and let $g\geq 0$ be its genus. Its
zeta function $Z(C,s)$ is defined (first for $s\in\Cc$ with $\Reel(s)$
large enough) by either of the equivalent expressions
$$
Z(C,s)=\exp\Bigl(\sum_{n\geq 1}{\frac{|C(\Fp_{q^n})|}{n}q^{-ns}}\Bigr)
=\prod_{\stacksum{\text{$x$ closed}}{\text{point in $C$}}}
{(1-N(x)^{-s})^{-1}}
$$
and it is well-known (as proved by F.K. Schmidt) that it can be
expressed as
$$
Z(C,s)=\frac{L(C,s)}{(1-q^{-s})(1-q^{1-s})}
$$
where $L(C,s)=P_C(q^{-s})$ for some polynomial $P_C(T)\in\Zz[T]$ of
degree $2g$. This polynomial (which is also called the $L$-function of
$C/\Fp_q$) may be factored as
$$
P_C(T)=\prod_{1\leq j\leq 2g}{(1-\alpha_j T)},
$$
where the ``roots'' (or inverse roots, really) $\alpha_j$, $1\leq
j\leq 2g$, satisfy $|\alpha_j|=\sqrt{q}$, as proved by Weil. This is
well-known to be the analogue of the Riemann Hypothesis, as we recall:
writing
$$
\alpha_j=q^{w_j}e(\theta_j),\quad\text{ where } w_j,\ \theta_j\in\Rr,\quad
e(z)=e^{2i\pi z},
$$
implies that the zeros $\rho$ of $L(C,s)$ are given by
$$
\rho=w_j+\frac{2\pi i \theta_j}{\log q}+\frac{2ik\pi }{\log q}
$$
for $k\in\Zz$, $1\leq j\leq 2g$. So Weil's result
$|\alpha_j|=\sqrt{q}$ corresponds to $w_j=1/2$, hence to
$\Reel(\rho)=1/2$ for any zero $\rho$ of $L(C,s)$.
\par
It is clearly of interest to investigate the possible linear relations
among those zeros as an analogue of the conjectures of linear
independence for ordinates of zeros of Dirichlet $L$-functions. Note
however that if we allow all imaginary parts, many ``trivial''
relations come from the fact that, e.g., the $\theta_j+k$, $k\in\Zz$,
are $\Qq$-linearly dependent. One must therefore consider $\theta_j$
up to integers, and the simplest way to do this is to consider
multiplicative relations
$$
\prod_{1\leq j\leq 2g}{e(n_j\theta_j)}=1
$$
with $n_j\in\Qq$ or, raising to a large power to eliminate the
denominator, relations
$$
\prod_{1\leq j\leq 2g}{\Bigl(\frac{\alpha_j}{\sqrt{q}}\Bigr)^{n_j}}=1
$$
with $n_j\in\Zz$. This, in fact, also detects $\Qq$-linear
dependencies among the components of the vector
$$
(1,\theta_1,\ldots, \theta_{2g})
$$
of size $2g+1$ (which is important for later applications).
\par
We will indeed study this problem, but at the same time we will
consider another independence question which seems fairly natural,
even if no particular analogue over number fields suggests itself: are
the $\alpha_j$, or the $1/\alpha_j$, linearly independent over $\Qq$?
\footnote{\ In fact, what we will prove about this will be helpful in
  one step of the study of the multiplicative case.}
\par
In the multiplicative case, it is immediately clear that we have to
take into account the functional equation
$$
L(C,s)=q^{g(1-2s)}L(C,1-s),
$$
which may be interpreted as stating that for any $j$, $q/\alpha_j$ is
also among the inverse roots. In particular, except if
$\alpha_j=\pm\sqrt{q}$, there are identities $\alpha_j\alpha_k=q$ with
$j\not=k$, leading to multiplicative relations of the form
$$
\alpha_j\alpha_k=\alpha_{j'}\alpha_{k'}
$$
(this is similar to the fact that a root $1/2+i\gamma$ of $L(s,\chi)$,
for a Dirichlet character $\chi$, gives a root $1/2-i\gamma$ of
$L(s,\bar{\chi})$, which leads to the restriction of the Grand
Simplicity Conjecture to non-negative ordinates of zeros). Hence the
most natural question is whether those ``trivial'' relations are the
only multiplicative relations.
\par
Finally, since dealing with a single curve seems still far away of
this Grand Simplicity Hypothesis, which involves all Dirichlet
$L$-functions, an even more natural-looking analogue would be to ask
the following: given a family of curves, interpreted as an algebraic
family $\mathcal{C}\ra U$ of curves of genus $g$ over some parameter
variety $U/\Fp_q$, what (if any) multiplicative relations can exist
among the $\alpha_{j}(t)/\sqrt{q}$ which are the inverse roots of the
polynomials $P_{C_t}(T)$, for \emph{all} $t\in U(\Fp_q)$?
\par
\medskip 
\par
We will prove in this paper some results which give evidence that this
type of independence holds. Of course, for a fixed curve, it might
well be that non-trivial relations do hold among the roots (see
Section~\ref{sec-numerics} for examples). However, looking at suitable
algebraic families, we will show that for \emph{most} curves $C_t$,
$t\in U(\Fp_q)$, their zeros and inverse zeros are as independent as
possible, both additively and multiplicatively.
\par
The first idea that may come to mind (along the lines
of~\cite{katz-sarnak}) is to use the fact that the set of matrices in
a compact group such as $SU(N,\Cc)$ or $USp(2g,\Cc)$ for which the
eigenvalues satisfy non-trivial relations is of measure zero (for the
natural measure, induced from Haar measure), and hope to apply
directly Deligne's Equidistribution Theorem, which states that after
taking suitable limits, the zeros of families of polynomials $P_{C_t}$
become equidistributed with respect to this measure.  However, the
sets in question, though they are measure-theoretically insignificant,
are also dense in the corresponding group, and this means
equidistribution does not by itself guarantee the required result. So
instead of this approach, we will use more arithmetic information on
the zeta functions.\footnote{\ Note however that for \emph{multiplicative}
relations, which are in a sense the most interesting, one can apply
Deligne's Theorem, after some preliminary work involving the specific
properties of the eigenvalues, see Section~\ref{sec-frob-tori}.}
\par
Here is now a sample statement, where we can easily give concrete
examples.  We use the following notation: given a finite family
$\uple{\alpha}=(\alpha_j)$ of non-zero complex numbers, we write
$\vect{\uple{\alpha}}$ for the $\Qq$-vector subspace of $\Cc$
generated by the $\alpha_j$, and $\mult{\uple{\alpha}}$ for the
multiplicative subgroup of $\Cc^{\times}$ generated by the
$\alpha_j$. For an algebraic curve $C$ over a finite field (resp. for
finitely many curves $\uple{C}=(C_1,\ldots, C_k)$ over a common base
field), we denote by $\zeros(C)$ the multiset of inverse zeros of
$P_C(T)$ (resp. by $\zeros(\uple{C})$ the multiset of inverse zeros of
the product $P_{C_1}\cdots P_{C_k}$), and similarly with $\nzeros(C)$
and $\nzeros(\uple{C})$ for the multisets of normalized inverse zeros
$\alpha/\sqrt{q}$.

\begin{proposition}\label{pr-1}
  Let $f\in\Zz[X]$ be a squarefree monic polynomial of degree $2g$,
  where $g\geq 1$ is an integer. Let $p$ be an odd prime such that $p$
  does not divide the discriminant of $f$, and let $U/\Fp_p$ be the
  open subset of the affine $t$-line where $f(t)\not=0$. Consider the
  algebraic family $\mathcal{C}_f\ra U$ of smooth projective
  hyperelliptic curves of genus $g$ given as the smooth projective
  models of the curves with affine equations
  $$
  C_t\,:\, y^2=f(x)(x-t),\quad\text{ for } t\in U.
  $$
  \par
  Then for any extension $\Fp_q/\Fp_p$, we have
\begin{equation}
  |\{t\in U(\Fp_q)\,\mid\, \text{there is a non-trivial linear
    relation among $\zeros(C)$} \}|\ll q^{1-\gamma^{-1}}(\log
  q)\label{eq-add},
\end{equation}
\begin{multline}
  |\{t\in U(\Fp_q)\,\mid\, \text{there is a non-trivial multiplicative }\\
  \text{ relation among $\nzeros(C)$} \}|\ll
  q^{1-\gamma^{-1}}(\log q),
\label{eq-mul}
\end{multline}
where $\gamma=4g^2+2g+4>0$, the implied constants depending only on $g$.
\end{proposition}

In order to explain precisely the meaning of the statements, and to
state further generalizations more concisely, we introduce the
following notation: for any finite set $M$ of complex numbers, we define
\begin{gather}
\reladd{M}=\{(t_{\alpha})\in\Qq^{M}\,\mid\, 
\sum_{\alpha\in M}{t_{\alpha}\alpha}=0\},
\label{eq-def-reladd}\\
\relmul{M}=\{(n_{\alpha})\in\Zz^{M}\,\mid\, 
\prod_{\alpha\in M}{\alpha^{n_{\alpha}}}=1\},
\label{eq-def-relmul}
\end{gather}
the additive relation $\Qq$-vector space and multiplicative relation
group, respectively. Note $\relmul{M}$ is a free abelian group.
\par
Then, tautologically, the condition in~(\ref{eq-add}) for a given
curve may be phrased equivalently as
$$
\reladd{\zeros(C)}=0,\text{ or }\dim_{\Qq}{\vect{\zeros(C)}}<2g,
\text{ or }
\vect{\zeros(C)}\simeq \Qq^{2g},
$$
and the qualitative content of~(\ref{eq-add}) is that this holds for
most values of $t$.
\par
The interpretation of~(\ref{eq-mul}) needs more care because
of the ``trivial'' multiplicative relations among the
$\tilde{\alpha}\in\nzeros(C_t)$. Precisely, from the functional
equation, it follows that we can arrange the $2g$ normalized roots
$\tilde{\alpha}=\alpha/\sqrt{q}$ in $g$ pairs of inverses
$(\tilde{\alpha},\tilde{\alpha}^{-1})$, so that the multiplicative
subgroup $\mult{\nzeros(C_t)}\subset \Cc^{\times}$ is of rank $\leq
g$. For $M=\nzeros(C_t)$, this corresponds to the inclusion
\begin{equation}\label{eq-reltriv}
\{(n_{\tilde{\alpha}})\in \Zz^{M}\,\mid\,
n_{\tilde{\alpha}}-n_{\tilde{\alpha}^{-1}}=0\}\subset \relmul{M}.
\end{equation}
\par
Denote by $\reltriv{M}$ the left-hand abelian group (which makes sense
for any $M\subset \Cc^{\times}$ stable under inverse), and let
$\relmult{M}=\relmul{M}/\reltriv{M}$ (the group of non-trivial
relations). The interpretation of~(\ref{eq-mul}) is that most of the
time, there is equality:
\begin{equation}\label{eq-trivial-mul}
\relmul{\nzeros(C_t)}=\reltriv{\nzeros(C_t)},\quad\text{ or }\quad
\relmult{\nzeros(C_t)}=0,
\end{equation}
(or, in fact, simply $\mult{\nzeros(C_t)}\simeq \Zz^{g}$; this is
because if $\mult{\nzeros(C_t)}$ is of rank $g$, comparing ranks
implies that $\reltriv{\nzeros(C_t)}$ is of finite index in
$\relmul{\nzeros(C_t)}$, and the former is easily seen to be saturated
in $\Qq^{\nzeros(C_t)}$, so it is not a proper finite index subgroup
of a subgroup of $\Zz^{\nzeros(C_t)}$).
\par
Moreover, yet another interpretation is the following. Assume still
that $M\subset \Cc^{\times}$ is stable under inverse and of even
cardinality $2g$; order its elements in some way so that
$$
M=\{\tilde{\alpha}_1,\ldots, \tilde{\alpha}_g,
\tilde{\alpha}_1^{-1},\ldots, \tilde{\alpha}_g^{-1}\},
$$
and write $\tilde{\alpha}_j=e(\theta_{j})$, with $0\leq \theta_{j}<
1$. Then $\relmult{M}=0$ if and only if the elements
$(1,\theta_1,\ldots, \theta_g)$ are $\Qq$-linearly
independent. Indeed, assuming the former, if we have a relation
$$
t_0+\sum_{1\leq j\leq g}{t_j \theta_j}=0
$$
with $(t_0,t_1,\ldots, t_g)\in\Qq^{g+1}$, multiplying by a common
denominator $\Delta$ and exponentiating leads to
$$
\prod_{1\leq j\leq g}{\tilde{\alpha}_j^{n_j}}=1
$$
where $n_j=\Delta t_j \in\Zz$. This implies that $(n_1,\ldots,
n_g,0,\ldots, 0)\in \relmult{M}=\reltriv{M}$, and by the
definition~(\ref{eq-reltriv}), we deduce $n_j=0$, $1\leq j\leq g$, and
then $t_j=0$ for all $j$. The converse is also easy.

\begin{remark}
  In the spirit of the previous remark concerning Deligne's
  Equidistribution Theorem, note that the result may be also
  interpreted (though this is much weaker) as giving instances of the
  convergence of $\mu_n(A)$ to $\mu(A)$, where $\mu_n$ is the average
  of Dirac measures associated to the normalized geometric Frobenius
  conjugacy classes in $USp(2g,\Cc)$ for $t\in \Fp_{q^n}$,
  $f(t)\not=0$, while $\mu$ is the probability Haar measure on
  $USp(2g,\Cc)$ and $A$ is the set of unitary symplectic matrices with
  eigenangles which are non-trivially additively or multiplicatively
  dependent (so, in fact, $\mu(A)=0$).
\end{remark}

\par
\medskip 
\par
Our second result encompasses the first one and is a first step
towards independence for more than one curve. Again we state it for
the concrete families above.

\begin{theorem}\label{th-2}
  Let $f\in\Zz[X]$ be a squarefree monic polynomial of degree $2g$,
  where $g\geq 1$ is an integer. Let $p$ be an odd prime such that $p$
  does not divide the discriminant of $f$, and let $U/\Fp_p$ be the
  open subset of the affine line where $f(t)\not=0$. Let
  $\mathcal{C}_f\ra U$ be the family of hyperelliptic curves defined
  in Proposition~\ref{pr-1}.
  \par
  Let $k\geq 1$. For all finite fields $\Fp_q$ of characteristic $p$,
  and for all $k$-tuples $\uple{t}=(t_1,\ldots, t_k)\in U(\Fp_q)^k$,
  denote $\uple{C}_{\uple{t}}=(C_{t_1},\ldots, C_{t_k})$.
Then we have
\begin{gather*}
  |\{\uple{t}\in U(\Fp_q)^k\,\mid\,
  \reladd{\zeros(\uple{C}_{\uple{t}})}\not=0 \}|\ll
  c^kq^{k-\gamma^{-1}}(\log q),
  \\
  |\{\uple{t}\in U(\Fp_q)^k\,\mid\,
  \relmult{\nzeros(\uple{C}_{\uple{t}})}\not=0 \}|\ll
  c^kq^{k-\gamma^{-1}}(\log q),
\end{gather*}
where $\gamma=29kg^2>0$ and $c\geq 1$ is a constant depending only on
$g$.  In both estimates, the implied constant depends only on $g$.
\end{theorem}

\begin{remark}
An important warning is to not read too much in this: the
dependency of $\gamma$ on $k$ means the result is trivial for $k$
unless we have (roughly speaking)
$$
c^{-k}q^{1/\gamma}\ra +\infty,\quad\text{i.e.}\quad
\frac{\log q}{29kg^2}-k\log c\ra +\infty,
$$
which means essentially (for fixed $g$) that $k=o(\sqrt{\log
  q})$. However, it leads to non-trivial results for any fixed $k$,
and even for $k$ growing slowly as a function of $q\ra +\infty$, and
in this respect it is already quite interesting. Also, note that the
``exceptional set'' of $k$-tuples trivially contains those $\uple{t}$
where two coordinates coincide; there are $\gg q^{k-1}$ of them, and
those ``diagonals'' would have to be excluded if one were to try to go
beyond such a bound.
\end{remark}

\begin{remark}
  We indicate the type of connections with distribution properties
  that arise. Those are of independent interest, and they show clearly
  the analogy with the discussion of the Chebychev Bias, in particular
  why the independence issues appear naturally there (compare with the
  arguments in~\cite[\S 2,\S 3]{rubinstein-sarnak}). 
\par
Let $C/\Fp_q$ is any (smooth, projective, geometrically connected)
algebraic curve of genus $g$, and choose $g$ inverse roots
$\alpha_{j}$ of the $L$-function of $C$, $1\leq j\leq g$, so that
$$
P_{C}(T)=\prod_{1\leq j\leq  g}
{(1-\alpha_{j}T)(1-\alpha_{j}^{-1}qT)},
$$
and write $\alpha_{j}=\sqrt{q}e(\theta_{j})$ as before.  For any
$n\geq 1$, the number of points in $C(\Fp_{q^n})$ is given by
$$
|C(\Fp_{q^n})|=
q^n+1-
\sum_{1\leq j\leq g}
{\Bigl(\alpha_{j}^n+\frac{q^n}{\alpha_{j}^{n}}\Bigr)}=
q^n+1-2q^{n/2}\sum_{1\leq j\leq g}{\cos 2\pi n\theta_{j}}.
$$
\par
If $C/\Fp_q$ is such that $\relmult{\nzeros(C)}=0$, we know that the
$g+1$ numbers $1$ and $(\theta_{j})_{i,j}$ are $\Qq$-linearly
independent. Hence, by Kronecker's theorem, the sequence
$$
(2\pi n\theta_{1},\ldots,2\pi n\theta_{g})\in (\Rr/2\pi\Zz)^g
$$
becomes equidistributed in $(\Rr/2\pi \Zz)^{g}$ as $n\ra +\infty$,
with respect to the Lebesgue measure on the torus. It follows that
$$
\frac{|C(\Fp_{q^n})|-(q^n+1)}{2q^{n/2}}
$$
becomes distributed like the image of Lebesgue measure under the map 
$$
\varphi\,:\,
\begin{cases}
  (\Rr/2\pi\Zz)^{g}\ra \Rr\\
  (\theta_1,\ldots, \theta_{g})\mapsto \cos
  \theta_1+\cdots+\cos\theta_g.
\end{cases}
$$
\par
This distribution is in fact not unexpected: we have the well-known
spectral interpretation
$$
\frac{|C(\Fp_{q^n})|-(q^n+1)}{2q^{n/2}}=\Tr(F^n)
$$
for $n\geq 1$, where $F\in USp(2g)$ is the unitarized Frobenius
conjugacy class of $C$. A remarkable result due to
E. Rains~\cite{rains} states that, for $n\geq 2g$, the eigenvalues of
a Haar-distributed random matrix in $USp(2g,\Cc)$ are distributed
\emph{exactly} like $g$ independent points uniformly distributed on
the unit circle, together with their conjugates. In particular, the
limit distribution above is therefore the distribution law of the
trace of such a random matrix.
\par
Similarly, let now $(C_1,C_2)$ be a pair of algebraic curves (smooth,
projective, geometrically connected) of common genus $g\geq 1$ over
$\Fp_q$, for which $\relmult{\nzeros(C_{1},C_{t})}=0$ -- for instance
any of the pairs given by Theorem~\ref{th-2} with $k=2$.  Write
$\alpha_{i,j}$, $\theta_{i,j}$ for the inverse roots and arguments as
above for $C_i$. 
\par
We compare the number of points on $C_1$ and $C_2$: we have
\begin{equation}\label{eq-diff-nb}
\frac{|C_1(\Fp_{q^n})|-|C_2(\Fp_{q^n})|}{q^{n/2}}
=2\sum_{1\leq j\leq g}{(\cos 2\pi n\theta_{2,j}-\cos 2\pi n \theta_{1,j})},
\end{equation}
\par
The assumption that $\relmult{\nzeros(C_{1},C_{2})}=0$ gives now that
the $2g+1$ numbers $1$ and $(\theta_{i,j})_{i,j}$ are $\Qq$-linearly
independent, and thus the sequence
$$
(2\pi n\theta_{2,1},\ldots,2\pi n\theta_{2,g},
2\pi n\theta_{1,1},\ldots,2\pi n\theta_{1,g})
$$
becomes equidistributed in $(\Rr/2\pi \Zz)^{2g}$ as $n\ra +\infty$
with respect to the Lebesgue measure on the $2g$-dimensional
torus. So, the right-hand side of~(\ref{eq-diff-nb}) becomes
equidistributed as $n\ra +\infty$ with respect to the image measure of
the Lebesgue measure $d\theta$ by the same map as above, with $2g$
angles instead of $g$ (since the cosine is an even function, it and
its opposite have the same distribution).
\par
Let $\mu_g$ be this measure, so that we have in particular, for any
$a<b$, the limit
$$
\frac{1}{N}
\Bigl|\Bigl\{
n\leq N\,\mid\,
a<\frac{|C_2(\Fp_{q^n})|-|C_1(\Fp_{q^n})|}{q^{n/2}}<b
\Bigr\}\Bigr|
\ra \int_{a}^b{d\mu_g(t)}
$$
as $N\ra +\infty$, and as a special case
$$
\frac{1}{N}
\Bigl|\Bigl\{
n\leq N\,\mid\,
|C_2(\Fp_{q^n})|<|C_1(\Fp_{q^n})|
\Bigr\}\Bigr|
\ra \frac{1}{2}
$$
as $N\ra +\infty$. This (since the assumption on $C_{1}$ and $C_{2}$
is ``almost always true'') means that there is typically no ``bias''
that can lead to the number of points on $C_1$ being larger than that
on $C_2$ when we look at extension fields of $\Fp_q$.
\par
Furthermore, we can clearly interpret $\mu_g$ as the probability law
of a sum
$$
Y_g=2\cos 2\pi X_1+\cdots +2\cos 2\pi X_{2g}
$$
of $2g$ independent random variables $2\cos 2\pi X_j$, $1\leq j\leq
2g$, where each $X_j$ is uniformly distributed on $[0,1]$ (there is no
minus sign since the cosine and its opposite have the same
distribution for uniform arguments). The characteristic function (in
other words, Fourier transform) of such a random variable is given by
$$
\varphi_g(t)=\expect(e^{itY_g})=\Bigl(
\int_0^1{e^{2it\cos 2\pi\theta}d\theta}\Bigr)^{2g}=J_0(2t)^{2g},
$$
where $J_0$ is the standard Bessel function. Furthermore, since
$$
\expect(Y_g)=0,\quad\quad 
\expect(Y_g^2)=2g\expect( (2\cos 2\pi X_1)^2) = 4g,
$$
the Central Limit Theorem implies that $Y_g/2\sqrt{g}$ converges in
law, as $g\ra +\infty$, to a standard Gaussian random variable with
variance $1$. This means that, for curves $C_1$ and $C_2$ of large
genus $g$, the further normalized difference
$$
\frac{|C_2(\Fp_{q^n})|-|C_1(\Fp_{q^n})|}{2q^{n/2}\sqrt{g}}
$$
will be distributed approximately like a standard Gaussian.
\par
It would be interesting to know what other limiting distributions can
occur for pairs of algebraic curves where there are non-trivial
relations (such as those in Section~\ref{sec-numerics}).
\end{remark}

As far as relating a result like Theorem~\ref{th-2} to the Grand
Simplicity Conjecture, even though the statement itself provides no
direct evidence, the main point is in the method of proof, which can
be interpreted as linking the problem with Random Matrix models for
families of $L$-functions. The point is that the crucial input to
apply the sieve for Frobenius, which is the main tool, is the fact
that the families of curves considered have large (symplectic)
monodromy, which in the Katz-Sarnak philosophy is the analogue of the
conjectured existence of ``symmetry types'' for families of
$L$-functions such as Dirichlet characters (precisely, the latter are
supposed to have unitary symmetry type, which is slightly
different). We refer to~\cite{rmt} for a survey of recent developments
in the area of Random Matrix models of $L$-functions, and for
discussion of the evidence available.
\par
The idea of the proofs is, roughly, to first show that a certain
maximality condition on the Galois group of the splitting field of an
individual set of zeros implies the required independence (see
Section~\ref{sec-algebraic}, which uses methods developed by Girstmair
to analyze relations between roots of algebraic equations). Then we
apply the sieve for Frobenius of the author (see~\cite{k1}
and~\cite[\S 8]{lsieve}) to check that most $\uple{C}_{\uple{t}}$
satisfy this criterion (as can be guessed from the statement of
Theorem~\ref{th-2}, the main novel issue in applying the sieve is the
need for some care in arguing uniformly with respect to $k$.)  One can
then see this type of argument as providing some kind of answer to the
question asked by Katz (see~\cite[End of Section 1]{katz-irred}) of
what could be a number field analogue of the irreducibility of zeta
functions of curves, or of other (polynomial) $L$-functions over
finite fields.
\par
After the first version of this paper was completed, along the lines
of the previous paragraph, N. Katz suggested to look at the
implications of the theory of Frobenius tori of Serre for this type of
questions.  It turns out that, indeed, one can use this theory (in the
version described by Cheewhye Chin~\cite{chin}) to get a different
proof of the multiplicative independence of the zeros, for fixed $k$
at least (in the setting of Theorem~\ref{th-2}).  The large monodromy
assumption remains essential, but the analytic argument is a bit
simpler, since one can use a uniform effective version of the
Chebotarev density theorem instead of the large sieve.  This, however,
does not significantly improve the final estimates. We sketch this
approach in Section~\ref{sec-frob-tori}. We have chosen to not remove
the earlier one because the sieve for Frobenius leads to added
information which may be useful for other purposes (e.g., the linear
independence of the roots is not controlled by Frobenius tori, see
Remark~\ref{rm-frob-tori-lindep}), and because it is (in some sense)
more elementary and accessible to analytic number theorists. For
instance, if we look at elements of $Sp(2g,\Zz)$ obtained by random
walks on such a discrete group, the approach based on Frobenius tori
would not be available to show that the probability of existence of
relations between eigenvalues of those matrices goes exponentially
fast to $0$, but it is an easy consequence of the large sieve
of~\cite[\S 7]{lsieve} and the results of Section~\ref{sec-algebraic}.
\par
We provide general versions of independence statements for any family
which has large (symplectic) monodromy.  Analogues for other symmetry
types are also easy to obtain; this is particularly clear from the
point of view of Frobenius tori, but the sieve for Frobenius can also
be adapted (see F. Jouve's thesis~\cite{jouve} for the case of
``big'' orthogonal monodromy).
\par
\medskip

\textbf{Acknowledgment.} Thanks to N. Katz for pointing out the
relevance of the work of Serre on Frobenius tori to questions of
multiplicative independence of Frobenius eigenvalues. The work of the
author was partially supported by the A.N.R through the ARITHMATRICS
project.
\par
\medskip

\textbf{Notation.} As usual, $|X|$ denotes the cardinality of a set,
$\mathfrak{S}_g$ is the symmetric group on $g$ letters, $\Fp_q$ is a
field with $q$ elements. By $f\ll g$ for $x\in X$, or $f=O(g)$ for
$x\in X$, where $X$ is an arbitrary set on which $f$ is defined, we
mean synonymously that there exists a constant $C\geq 0$ such that
$|f(x)|\leq Cg(x)$ for all $x\in X$. The ``implied constant'' is any
admissible value of $C$. It may depend on the set $X$ which is always
specified or clear in context.  On the other hand, $f=o(g)$ as $x\ra
x_0$ means that $f/g\ra 0$ as $x\ra x_0$. 
\par
An algebraic variety is meant to be a reduced, separated scheme of
finite type, and most of those occurring will be affine. For $V/\Fp_q$
an algebraic variety over a finite field, $\nu\geq 1$ and $t\in
V(\Fp_{q^{\nu}})$, we write $\frob_{q^{\nu},t}$ for the geometric
Frobenius conjugacy class at $t$ relative to the field
$\Fp_{q^{\nu}}$; when $\nu$ is fixed, we simply write $\frob_t$. For a
field $k$, we write $\bar{k}$ for an algebraic closure of $k$, and for
an algebraic variety $X$ over $k$, we write $\bar{X}$ for $X\times_k
\bar{k}$, and we denote by $\eta_X$ a geometric $\bar{k}$-valued point
of $X$; whenever morphisms between fundamental groups are mentioned,
the geometric points are assumed to be chosen in compatible fashion.

\section{An algebraic criterion for independence}
\label{sec-algebraic}

Let $g\geq 1$ be a fixed integer, and let $W_{2g}$ be the finite group
of order $2^gg!$ which is described (up to isomorphism) by any of the
following equivalent definitions:
\par
-- it is the group of permutations of a finite set $M$ of order $2g$
which commute with a given involution $c$ on $M$ without fixed points:
$$
\sigma(c(\alpha))=c(\sigma(\alpha))\quad\text{for all $\alpha\in M$}\ ;
$$
we write usually $c(\alpha)=\bar{\alpha}$, so that
$\overline{\sigma(\alpha)}=\sigma(\bar{\alpha})$. 
\par
-- given a set $M$ with $2g$ elements which is partitioned in a set
$N$ of $g$ couples $\{x,y\}$, $W_{2g}$ is the subgroup of the group of
permutations of $M$ which permute the set of pairs $N$; as an example,
we can take
$$
M=\{-g,\ldots, -1,1,\ldots ,g\}\subset \Zz
$$ 
with the pairs $\{-i,i\}$ for $1\leq i\leq g$, and then the condition
for a permutation $\sigma$ of $M$ to be in $W_{2g}$ is that
$$
\sigma(-i)=-\sigma(i),\quad\text{ for all $i$, $1\leq i\leq g$}.
$$
\par
-- it is the semi-direct product $\mathfrak{S}_g\ltimes \{\pm 1\}^g$
where $\mathfrak{S}_g$ acts on $\{\pm 1\}^g$ by permuting the
coordinates.
\par
-- it is the subgroup of $GL(g,\Qq)$ of matrices with entries in
$\{-1,0,1\}$, where one entry exactly in each row and column is
non-zero.\footnote{\ As explained in~\cite{bdeps}, except for seven
  values of $g$, this is in fact a finite subgroup of $GL(g,\Qq)$ with
  maximal order.}
\par
-- finally, it is the Weyl group of the symplectic group $Sp(2g)$,
i.e, the quotient $N(T)/T$ where $T\subset Sp(2g)$ is a maximal torus
(although this can be seen as the ``real'' reason this group occurs in
our context, it is not at all necessary to know the details of this
definition, or how it relates to the previous ones, to understand the
rest of this paper).
\par
Note that the second definition provides a short exact sequence
\begin{equation}\label{eq-exact-seq}
1\ra \{\pm 1\}^g\ra W_{2g}\ra \mathfrak{S}_g\ra 1.
\end{equation}
\par
We will use mostly the first two definitions, the equivalence of which
is particularly easy, indicating what is the set $M$ and/or involution
$c$ under consideration. We let $N$ be the quotient of $M$ modulo the
equivalence relation induced by $c$ (with $\alpha\sim \bar{\alpha}$;
this is the same as the set $N$ of the second definition).
\par
\medskip
\par
We now state some properties of the group $W_{2g}$, which we assume to
be given with some set $M$ and set $N$ of couples on which $W_{2g}$
acts, as in the second definition. For a given $\alpha\in M$, we write
$\bar{\alpha}$ for the unique element such that
$\{\alpha,\bar{\alpha}\}\in N$.
\par
We let $F(M)=\Qq^{M}$ be the $\Qq$-vector space generated by $M$, with
canonical basis $(f_{\alpha})_{\alpha\in M}$, and we consider $F(M)$ as
given with the associated permutation representation of $W_{2g}$.

\begin{lemma}\label{lm-one-w}
  Let $g\geq 2$ be any integer, $W_{2g}$, $M$, $N$ and $F(M)$ as
  before. Then
\par
\emph{(1)} The group $W_{2g}$ acts transitively on $M$, and acts on
$M\times M$ with three orbits:
\begin{gather*}
  \Delta=\{(\alpha,\alpha)\,\mid\, \alpha\in M\},\quad\quad
  \Delta_c=\{(\alpha,\bar{\alpha})\,\mid\, \alpha\in M\},\\
  O=\{(\alpha,\beta)\,\mid\, \alpha\not=\beta,
  \ \bar{\alpha}\not=\beta\}.
\end{gather*}
\par
\emph{(2)} The representation of $W_{2g}$ on $F(M)$ decomposes as
the direct sum
$$
F(M)=\charfun\oplus G(M)\oplus H(M)
$$
of the three subspaces defined by
\begin{gather*}
  \charfun=\Qq \psi\subset F(M),\text{ where } \psi=\sum_{\alpha\in
    M}{f_{\alpha}},
  \\
  G(M)=\Bigl\{\sum_{\alpha\in M}{t_{\alpha}f_{\alpha}}\in F(M)\,\mid\,
  t_{\alpha}-t_{\bar{\alpha}}=0,\ \alpha\in M,\text{ and }
  \sum_{\alpha\in M}{t_{\alpha}}=0 \Bigr\},\\
  H(M)=\Bigl\{\sum_{\alpha\in M}{t_{\alpha}f_{\alpha}}\in F(M)\,\mid\,
   t_{\alpha}+t_{\bar{\alpha}}=0,\ \alpha\in M\Bigr\},
\end{gather*}
which are absolutely irreducible representations of $W_{2g}$.
\end{lemma}

\begin{proof}
  (1) The transitivity of $W_{2g}$ on $M$ is clear. Furthermore, it is
  obvious that the sets $\Delta$, $\Delta_c$, $O$ form a partition of
  $M\times M$, and that $\Delta$ is the orbit of any fixed
  $(\alpha,\alpha)\in \Delta$ by transitivity.
\par
To check that $\Delta_c$ is also an orbit, fix some
$x_0=(\alpha_0,\bar{\alpha}_0)\in \Delta_c$, and let
$x=(\alpha,\bar{\alpha})\in \Delta_c$ be arbitrary. If $\sigma$ is any
element of $W_{2g}$ such that $\sigma(\alpha_0)=\alpha$, we have
$\sigma(\bar{\alpha})=\bar{\alpha}_0$, hence $\sigma(x_0)=x$.
\par 
There remains to look at $O$. First $O\not=\emptyset$ because
$g\geq 2$ (so that there exist $(\alpha, \beta)\in M\times M$ with
$\beta\notin \{\alpha,\bar{\alpha}\}$).  Using the fact that for any
$\gamma\not=\delta$ in $M$, there exists $\sigma\in W_{2g}$ such that
$\sigma(\gamma)=\delta$ and $\sigma$ acts as identity on
$M-\{\gamma,\bar{\gamma},\delta,\bar{\delta}\}$, it is clear that if
$y=(\alpha,\beta)\in O$, then all elements of $O$ of the form
$(\alpha,\gamma)$ are in the orbit of $y$, and so are all elements of
the form $(\gamma,\beta)$.
\par
So given $y_1=(\alpha,\beta)$ and $y_2=(\gamma,\delta)\in O$, we can find
$\sigma_1$ such that $\sigma(y_1)=(\alpha,\delta)$, then $\sigma_2$
such that 
$$
\sigma_1\sigma_2(\alpha,\beta)=\sigma_2(\alpha,\delta)=(\gamma,\delta)=y_2,
$$
so $O$ is a single orbit as desired.
\par
(2) Again, it is easily checked that $\charfun$, $G(M)$ and $H(M)$ are
$W_{2g}$-invariant subspaces of $F(M)$, and it suffices to check that the
representation $F(M)\otimes \Cc$ is a direct sum of three irreducible
components. This means we must show that
$$
\langle \chi,\chi\rangle=3
$$
where $\chi$ is the character of the representation of $W_{2g}$ on
$F(M)\otimes\Cc$, as $3$ can only be written as $1+1+1$ as sum of
squares of positive integers. This is a well-known consequence of (1):
since $\chi$ is real-valued (as character of a permutation
representation), we have $\langle \chi,\chi\rangle=\langle
\chi^2,1\rangle$; further, $\chi^2$ is the character of the
permutation representation of $W_{2g}$ on $M\times M$, and hence, as
for any permutation character, the inner product $\langle
\chi^2,1\rangle$ is the number of orbits of the action of $W_{2g}$ on
$M\times M$, which we saw is equal to $3$ (for these facts, see,
e.g.,~\cite[Exercise 2.6]{serre-rep}).
\end{proof}

\begin{remark}
  The first part of the lemma says that $W_{2g}$ does not act
  doubly-transitively on $M$, but is not so far from this, the orbit
  $O$ being of much larger size than the diagonal orbit $\Delta$ and
  $\Delta_c$ (the graph of the involution $c$ on $M$): we have
  $|\Delta|=|\Delta_c|=2g$, and $|O|=4g(g-1)$).
\par
On the other hand, we have $\dim \charfun=1$, $\dim G(M)=g-1$ and $\dim
H(M)=g$. If we select one element of each of the $g$ pairs in $N$ and
number them as $(\alpha_i,\bar{\alpha}_i)$ for $1\leq i\leq g$, then
bases of $\charfun$, $G(M)$ and $H(M)$ are given, respectively, by the
vectors
\begin{gather}
\sum_{\alpha\in M}{f_{\alpha}},
\\
  (f_{\alpha_i}+f_{\bar{\alpha}_i})-
  (f_{\alpha_{i+1}}+f_{\bar{\alpha}_{i+1}}),\quad 1\leq i\leq g-1,
\label{eq-basis-gm}\\
  f_{\alpha_i}-f_{\bar{\alpha}_i},\quad 1\leq i\leq g.
\label{eq-basis-hm}
\end{gather}
Note that we also obtain from the definitions of $\charfun$ and $G(M)$
that 
\begin{equation}\label{eq-trivial-dec}
\charfun\oplus G(M)=
\Bigl\{\sum_{\alpha\in M}{t_{\alpha} f_{\alpha}}\in F(M)\,\mid\,
t_{\alpha}=t_{\bar{\alpha}},\quad \alpha\in M\Bigr\}
\end{equation}
(which is none other than $\reltriv{M}$, as defined
in~(\ref{eq-reltriv}).)
\par
In terms of ``abstract'' representation theory, the three subspaces
are not hard to identify: notice first that both $\charfun$ and $G(M)$
are invariant under the subgroup $(\Zz/2\Zz)^g$ in the exact
sequence~(\ref{eq-exact-seq}), hence are representations of the
quotient $\mathfrak{S}_g$. It is clear that their direct sum is simply
the standard permutation representation of the symmetric group.  As
for $H(M)$, looking at the action on the basis~(\ref{eq-basis-hm}), one
finds that it is isomorphic to the representation given by the
embedding $W_{2g}\injecte GL(g,\Qq)$ of the last definition of
$W_{2g}$ (in particular, it is faithful).
\end{remark}

\begin{corollary}\label{cor-many-w}
  Let $k\geq 1$ be an integer and $W=W_{2g}\times\cdots\times W_{2g}$,
  the product of $k$ copies of $W_{2g}$, the $j$-th copy acting on
  $M_j$.  Consider the action of $W$ on the disjoint union
$$
M=\bigsqcup_{1\leq j\leq k}{M_j}
$$
where the $j$-th factor acts trivially on $M_i$ for $i\not=j$.  Let
$F(M)$ denote the permutation representation of $W$ on the $\Qq$-vector
space $\Qq^{M}$ of dimension $2kg$. Then $F(M)$ is $\Qq$-isomorphic to
the direct sum
$$
F(M)\simeq k\cdot \charfun\oplus \bigoplus_{1\leq j\leq k}{G_{j}}
\oplus \bigoplus_{1\leq j\leq k}{H_j}
$$
of geometrically irreducible representations of $W$, where $G_j$ is
the representation $G(M_j)$ of the previous lemma, $(\sigma_1,\ldots,
\sigma_k)$ acting as $\sigma_j$, and similarly $H_j$ is $H(M_j)$
acting through the $j$-th factor $W_{2g}$.
\end{corollary}

\begin{proof}
This is clear from Lemma~\ref{lm-one-w} and the definition of $M$.
\end{proof}

Continuing with an integer $k\geq 1$, we now assume that we have
polynomials $P_1$, \ldots, $P_k$ with coefficients in a field
$E\subset \Cc$ such that each of the splitting fields $K_i/E$ of $P_i$
has Galois group isomorphic to $W_{2g}$, acting by permutation on the
set $M_j$ of roots of $P_j$, and which are jointly linearly
independent so that the splitting field $K/E$ of the product
$$
P=P_1\cdots P_k\in E[X]
$$
has Galois group naturally isomorphic to $W=W_{2g}^k$. Note that this
implies in particular that the sets of roots of the polynomials $P_j$
are disjoint. Then the disjoint union $M$ of
Corollary~\ref{cor-many-w} can be identified with the set of all roots
of $P$.
\par
We have the $\Qq$-vector space $\vect{M}\subset \Cc$ generated by the
set of roots of $P$, and the multiplicative abelian group
$\mult{M}\subset \Cc^{\times}$, from which we may construct the
$\Qq$-vector space $\mult{M}\otimes_{\Zz}\Qq$. Using the Galois action
by permutation of the roots, those two vector spaces are themselves
representations of $W$, and moreover mapping each element of the
canonical basis of $F(M)=\Qq^{M}$ to the corresponding root, we have
natural $\Qq$-linear maps
$$
F(M)=\Qq^{M}\fleche{r_a} \vect{M},\quad\quad
F(M)=\Qq^{M}\fleche{r_m} \mult{M}\otimes \Qq,
$$
which are also maps of $W$-representations. By construction, we have
$$
\ker(r_a)=\reladd{M},\quad \ker(r_m)=\relmul{M}\otimes \Qq,
$$
where $\reladd{M}$ and $\relmul{M}$ are the relation groups defined
in~(\ref{eq-def-reladd}) and~(\ref{eq-def-relmul}).
Note that both $\reladd{M}$ and $\relmul{M}\otimes \Qq$ are
subrepresentations of the permutation representation $F(M)$.
\par
Thus we see that the problem of finding the possible relations among
roots of a polynomial is transformed into a problem of representation
theory (in the multiplicative case, one must also handle the possible
loss of information in taking the tensor product with $\Qq$: for
instance, $\relmul{-1}=2\Zz\subset \Zz$, and $1\in
\relmul{-1}\otimes\Qq$ although $(-1)^1\not=1$...).  This is in
essence Girstmair's method, see e.g.~\cite{girstmair} (notice that
there is nothing special in working with $W$-extensions in the above).
Since Corollary~\ref{cor-many-w} has described explicitly the
decomposition of $F(M)$ as sum of irreducible representations of $W$,
the theory of linear representations of finite groups shows that there
are very few possibilities for the subrepresentations $\reladd{M}$ and
$\relmul{M}\otimes \Qq$.

\begin{proposition}\label{pr-w2g-k}
  Let $k\geq 1$ and $g\geq 2$ be integers. Let $P_1$, \ldots, $P_k$ be
  polynomials satisfying the conditions above. With notation as above,
  in particular $P=P_1\cdots P_k$ and $M$ the set of zeros of $P$,
  assume in addition that for any pair of roots
  $(\alpha,\bar{\alpha})$, we have
  $\alpha\bar{\alpha}\in\Qq^{\times}$.
\par
\emph{(1)} We have
$$
\reladd{M}=\bigoplus_{1\leq j\leq k}{\reladd{M_j}},
$$
and for each $j$, we have either $\reladd{M_j}=0$, or
$\reladd{M_j}=\charfun$. The latter alternative holds if and only if
$$
\sum_{\alpha\in M_j}{\alpha}=0
$$
or equivalently if $\Tr_{K/E}(\alpha)=0$ for any $\alpha\in M_j$.
\par
\emph{(2)} We have
$$
\relmul{M}\otimes \Qq=\bigoplus_{1\leq j\leq k}{\relmul{M_j}\otimes\Qq}.
$$
Moreover, assume  that the rational number
$\alpha\bar{\alpha}\in\Qq$ is positive and independent of $\alpha$,
say equal to $m$. Then for $g\geq 5$ in the general case, and for
$g\geq 2$ if $m=1$, we have for each $j$ that
$$
\relmul{M_j}\otimes \Qq=
\begin{cases}
\charfun\oplus G(M_j)&\text{ if } m=1,\\
G(M_j)&\text{ otherwise.}
\end{cases}
$$
\end{proposition}

\begin{proof}
(1) From representation theory, we know that $\reladd{M}$ is the direct
  sum of some subset of the irreducible components of $F(M)$
  corresponding to the decomposition in
  Corollary~\ref{cor-many-w}. This isomorphism shows that $F(M)$
  decomposes as a direct sum over $j$ of representations $F(M_j)$
  depending on the $j$-th factor of $W$, each of which is given by
  Lemma~\ref{lm-one-w}. Accordingly, $\reladd{M}$ is the direct sum
  over $j$ of subrepresentations of $F(M_j)$. Those are
  representations of the $j$-th factor $W_{2g}$ extended by the
  identity to $W$, and tautologically, they correspond exactly to the
  relation space $\reladd{M_j}$ among zeros of $P_j$.
\par
To finish the proof of (1), it suffices therefore to treat each $P_j$
in turn, so we might as well assume $k=1$ and remove the subscript
$j$, using notation in Lemma~\ref{lm-one-w} (in particular, writing
now $M$ instead of $M_j$).  Noting that, for any $\alpha\in M$, the
relation $\Tr_{K/E}(\alpha)=0$ is equivalent with $\charfun\subset
\reladd{M}$, the claim then amounts to saying that $G(M)$ and $H(M)$ can
not occur in $\reladd{M}$.
\par
First, $G(M)\subset \reladd{M}$ means that 
\begin{equation}\label{eq-rel}
\sum_{\alpha}{t_{\alpha}\alpha}=0
\end{equation}
whenever $(t_{\alpha})\in\Qq^{M}$ sum to zero and satisfy
$t_{\alpha}-t_{\bar{\alpha}}=0$ for $\alpha\in M$. In particular,
fix a root $\alpha$ of $P$; we find that for any $\sigma\in W_{2g}$
with $\sigma(\alpha)\not=\alpha$, say $\sigma(\alpha)=\beta$, we have
$$
(\alpha+\bar{\alpha})-
(\beta+\bar{\beta})=
(\alpha+\bar{\alpha})-
\sigma(\alpha+\bar{\alpha})=0
$$
for all $\sigma\in W_{2g}=\Gal(K/\Qq)$ not fixing $\alpha$. Since the
last relation is trivially valid for $\sigma$ fixing $\alpha$ (hence
$\bar{\alpha}$), it follows that $\alpha+\bar{\alpha}\in \Qq$. From
the assumption $\alpha\bar{\alpha}\in\Qq^{\times}$, it follows that
$\Qq(\alpha)$ is a quadratic field. It must be the splitting field $K$
of the polynomial $P$, and hence this can not occur under the
conditions $g\geq 2$ and $\Gal(K/\Qq)=W_{2g}$.
\par
Similarly $H(M)\subset \reladd{M}$ means that~(\ref{eq-rel}) holds
whenever $(t_{\alpha})\in\Qq^{M}$ satisfy
$t_{\alpha}+t_{\bar{\alpha}}=0$. Using again a fixed root $\alpha$ of
$P$, we obtain in particular
\begin{equation}\label{eq-cons-hm}
\alpha-\bar{\alpha}=0
\end{equation}
which contradicts the fact that the elements $\alpha$ and
$\bar{\alpha}$ are distinct.
\par
(2) The proof of the direct sum decomposition for
$\relmul{M}\otimes\Qq$ is the same as that for additive relations, and
hence we are again reduced to the case $k=1$ (and we write $M$ instead
of $M_j$). We first show that $G(M)\subset \relmul{M}\otimes\Qq$ in all
cases. Indeed, considering the generators~(\ref{eq-basis-gm}) of
$G(M)$, it suffices to show that
$$
\frac{\alpha\bar{\alpha}}{\beta\bar{\beta}}=1
$$
for all $\alpha$ and $\beta$, and this is correct from our assumption
that $\alpha\bar{\alpha}$ is independent of $\alpha$. (Note the tensor
product with $\Qq$ means this is not \emph{equivalent} with
$G(M)\subset \relmul{M}\otimes\Qq$).
\par
Now we consider the consequences of the possible inclusion of the
subrepresentations $\charfun$, and $H(M)$ in $\relmul{M}\otimes\Qq$.
First, $\charfun\subset \relmul{M}\otimes\Qq$ means exactly that for
some integer $n\geq 1$, we have
$$
n\psi=\sum_{\alpha\in M}{n f_{\alpha}}\in \relmul{M},
$$
which is equivalent with
$$
\prod_{\alpha\in M}{\alpha^n}=\Bigl(\prod_{\alpha\in
  M}{\alpha}\Bigr)^n=(N_{K/E}(\alpha))^n=1,
$$
or in other words, $N_{K/E}(\alpha)$ is a root of unity. But the
assumption that $\alpha\bar{\alpha}=m$ be a positive rational number
independent of $\alpha$ implies that $N_{K/E}(\alpha)=m^g$, so
$\charfun\subset \relmul{M}\otimes\Qq$ if and only if $m=1$.
\par
It remains to exclude the possibility that $H(M)\subset
\relmul{M}\otimes\Qq$ to conclude the proof.  But instead
of~(\ref{eq-cons-hm}), this possibility implies now that, for some
integer $n\geq 1$, we have
$$
\alpha^{2n}=m^n\Bigl(\frac{\alpha^{2n}}{m^n}\Bigr)
=m^n\Bigl(\frac{\alpha}{\bar{\alpha}}\Bigr)^n=m^n.
$$
\par
Hence $K/\Qq$ would be the Kummer extension
$\Qq(\sqrt{m},\uple{\mu}_{2n})$, where $\uple{\mu}_{2n}$ is the group
of $2n$-th roots of unity. In particular, the Galois group of $K/E$
would be solvable, which is false for $W_{2g}$ if $g\geq 5$ (the
non-solvable group $A_g$ occurs as one composition factor). For $m=1$,
the Galois group would be abelian, which is not the case of $W_{2g}$
for $g\geq 2$.
\end{proof}

\begin{remark}
  Since there exist elements with trace zero generating a given number
  field, both cases of the alternative in (1) can occur. It should be
  clear however that $\reladd{M_j}=0$ is the ``most likely'', and we
  will see this at work in Section~\ref{sec-sieve}.
\end{remark}


\begin{remark}
  In~\cite[Prop. 2.1, Remark 2.2]{weilnbs}, we had proved for
  different purposes and using quite different methods a result which
  implied, as we remarked, that if the splitting field of the
  $L$-function of a curve $C/\Fp_q$ is $W_{2g}$, and if in addition
  the curve were ordinary (which can be interpreted as saying that the
  coefficient of $T^g$ of $P_C$ is not divisible by $p$), then the
  multiplicative group $\mult{\zeros(C)}$ is free of rank $g+1$. This
  is almost the same as the case $k=1$ of Proposition~\ref{pr-w2g-k},
  but it would be very inconvenient below to have to assume
  ordinarity. As explained by Milne~\cite[2.7]{milne}, the freeness of
  the group generated by the inverse roots also has consequences for
  the Tate conjecture.
\end{remark}


\begin{remark}
Since this may be useful in other investigations, we quote the
analogue of Proposition~\ref{pr-w2g-k} when $W_{2g}$ is replaced by
the symmetric group $\mathfrak{S}_n$, $n\geq 2$. The proof is easier
than the previous one (because the natural action of $\mathfrak{S}_n$
on sets of order $n$ is doubly transitive), and in fact is contained
in the works of Girstmair.
\par
\begin{proposition}
  Let $k\geq 1$ and $n\geq 2$ be integers. Let $P_1$, \ldots, $P_k$ be
  polynomials with rational coefficients of degree $n$ such that
  $P=P_1\cdots P_k$ has splitting field $K$ with Galois group
  $\mathfrak{S}_n^k$. Let $M$ be the set of complex roots of $P$,
  $M_j$ that of $P_j$.
\par
\emph{(1)} We have
$$
\reladd{M}=\bigoplus_{1\leq j\leq k}{\reladd{M_j}},
$$
and for each $j$, we have either $\reladd{M_j}=0$, or
$$
\reladd{M_j}=\charfun=\Qq\cdot \sum_{\alpha\in M_j}{\alpha},
$$
and the latter alternative holds if and only if, for any $\alpha\in
M_j$, we have $\Tr_{K/\Qq}(\alpha)=0$.
\par
\emph{(2)} We have
$$
\relmul{M}=\bigoplus_{1\leq j\leq k}{\relmul{M_j}},
$$
and for each $j$, $\relmul{M_j}$ is one of the following:
$$
0\ ,\quad\quad
m_j\Zz \cdot \sum_{\alpha\in M_j}{\alpha}\quad\quad
n_j\Zz^{M_j}\ ,\quad\quad
m'_j\cdot \Bigl\{
(n_{\alpha})\,\mid\, \sum_{\alpha}{n_{\alpha}}=0
\Bigr\},
$$
where $m_j\in \{1,2\}$, $n_j\in \{3,4,6\}$, $m'_j\in \{2,3\}$. The
third case holds when $M_j$ is the set of roots of unity of order
$n_j$.  The second case holds when the third one doesn't and
$N_{K/\Qq}(\alpha)=(-1)^{m_j-1}$, i.e., when the $\alpha\in M_j$ are
units, not roots of unity. The fourth case occurs when the two
previous do not, and $\alpha$ satisfies a Kummer equation
$\alpha^{m'_j}=\beta\in\Qq^{\times}$, $\beta$ not an $m'_j$-th power
of an integer.
\end{proposition}
\end{remark}

\section{The simplest case: proof of Proposition~\ref{pr-1} }
\label{sec-simple}

We start with a proof of Proposition~\ref{pr-1}, although it is
subsumed in Theorem~\ref{th-2}, because we can quote directly from
earlier results of the author on Galois groups of splitting fields of
numerators of the zeta functions in those families of curves (we
recall also that the first qualitative result on this topic is due to
Chavdarov~\cite{chavdarov}).  This means we can avoid setting up anew
the general sieve for Frobenius, and in particular we not need to
refer explicitly to the fairly sophisticated algebraic geometry which
is involved.
\par
Consider then a squarefree monic polynomial $f\in\Zz[X]$ of degree
$2g$ and an odd prime $p$ not dividing the discriminant of $f$. Let
$q\not=1$ be a power of $p$. For each $t\in \Fp_q$ with $f(t)\not=0$,
we consider the (smooth projective model of the) hyperelliptic curve
$$
C_t\,:\, y^2=f(x)(x-t),
$$
which is of genus $g$ so that the $L$-function $P_t\in \Zz[T]$ of
$C_t$, as defined in the introduction, has degree $2g$. 
\par
For a fixed $q$, we say that $t\in \Fp_q$ is \emph{special} if any one 
of the following condition holds:
\par
-- We have $f(t)=0$.
\par
-- The Galois group of the splitting field of $P_t$ is not isomorphic
to $W_{2g}$ (which is the largest it can be because of the functional
equation of the zeta function).
\par
-- The sum of the inverse roots $\alpha\in\zeros(C_t)$ is $0$.
\par
Then, under the assumptions stated, it follows from Theorem 8.1
in~\cite{lsieve} (see also~\cite[Th. 6.2]{k1}) that
$$
|\{t\in \Fp_q\,\mid\, t\text{ is special}\}|
\ll q^{1-\gamma^{-1}}(\log q),
$$
where $\gamma=4g^2+2g+4$ and the implied constant depends only on
$g$. More precisely, those results only deal with the first two
conditions (of which the second is of course the one which is
significant), but the simplest type of sieve (or rather uniform
Chebotarev density theorem) shows that
$$
|\{t\in \Fp_q\,\mid\,
\text{$f(t)\not=0$ and the sum of inverse roots of $P_t$ is zero}
\}|
\ll q^{1-\gamma^{-1}},
$$
simply because it is an algebraic condition on the coefficients of the
polynomial (see the proof of Theorem~\ref{th-2} for details in the
general case $k\geq 2$).
\par
Consider now any $t\in \Fp_q$  which is not special. We will show that
the roots of the zeta function of $C_t$ satisfy the two independence
conditions in Proposition~\ref{pr-1}, and this will finish the proof,
in view of the bound on the number of special parameters $t$.
\par
Because it is fixed, we drop the dependency on $t$ from the notation
from now on, unless this creates ambiguity.  The additive case is
clear from the first part of Proposition~\ref{pr-w2g-k} applied with
$k=1$, $m=q$ and
$$
P=T^{2g}P_t(T^{-1})\in\Zz[T]
$$
(which has the $\alpha\in \zeros(C_t)$ as roots), since the splitting
field $K$ of this polynomial is the same as that of $P_t$, hence its
Galois group is indeed $W_{2g}$, and the sum of the roots of $P$ is
non zero for $t$ not special, by the very definition.
\par
Now we come to the multiplicative independence of the normalized
inverse roots. Recall first that with $M=\nzeros(C_t)$, and involution
given by 
$$
\bar{\alpha}=c(\alpha)=\frac{1}{\alpha},
$$
the desired conclusion~(\ref{eq-trivial-mul}) can be rephrased as
$$
\relmul{ \nzeros(C_t) }=\{(n_{\tilde{\alpha}})\in \Zz^M\,\mid\,
n_{\tilde{\alpha}}-n_{\tilde{\alpha}^{-1}}=0\},
$$
and the left-hand side does contain the right-hand side, so only the
reverse inclusion is required.
\par
The elements of $M$ are roots of the polynomial
$$
Q_t=T^{2g}P_t(q^{-1/2}T^{-1})\in \Qq(\sqrt{q})[T],
$$
which creates a slight complication: if (as seems natural) we extend
scalars to $E=\Qq(\sqrt{q})$ to have $Q_t\in E[T]$, there is a
possibility that the Galois group of its splitting field (over $E$) is
not $W_{2g}$ anymore (e.g., when $\sqrt{q}$ is in the splitting field
of $P_t$). We deal with this by looking at the squares of the inverse
roots. 
\par
Let 
$$
M'=\{\tilde{\alpha}^2\,\mid\, \tilde{\alpha}\in M=\nzeros(C_t)\}=
\{\alpha^2/q\,\mid\, \alpha\in \zeros(C_t)\};
$$
the second expression shows that $M'\subset K=\Qq(\zeros(C_t))$, so
the field $F=\Qq(M')$ is a subfield of $K$. Its Galois group is the
group of those $\sigma\in \Gal(K/\Qq)$ which fix all $\alpha^2$ for
$\alpha\in \zeros(C_t)$, i.e., such that $\sigma(\alpha)\in
\{\alpha,-\alpha\}$ for all $\alpha$. If $\sigma\in \Gal(K/F)$ is not
the identity, there exists some $\alpha\in \zeros(C_t)$ such that
$\beta=\sigma(\alpha)$ is equal to $-\alpha$, and this leads to
$\alpha+\beta=0$, in particular to $\reladd{\zeros(C_t)}\not=0$. Since
this contradicts the previous observation that the elements of
$\zeros(C_t)$ are $\Qq$-linearly independent when $t$ is not special,
we have in fact $\Gal(K/F)=1$, and so $F=K$.
\par
We can now apply (2) of Proposition~\ref{pr-w2g-k}, with $k=m=1$ and
$P$ taken to be the polynomial with zeros $M'$, namely
$$
\prod_{\gamma\in M'}{(T-\gamma)}=
\prod_{\tilde{\alpha}\in M}{(T-\tilde{\alpha}^2)}\in
\Qq[T],
$$
with $F=K$ such that $\Gal(F/\Qq)=W_{2g}$, acting by permutation of
the set $M'$ with the involution
$$
c(\gamma)=\gamma^{-1},\quad\text{i.e.}\quad
c(\tilde{\alpha}^2)=\tilde{\alpha}^{-2}.
$$
\par
Since $\gamma c(\gamma)=1$ for all $\gamma\in M'$, we obtain
$$
\relmul{M'}\otimes_{\Zz}\Qq
=\charfun\oplus G(M')=
\{(n_{\gamma})\in \Qq^{M'}\,\mid\,
n_{\gamma}-n_{c(\gamma)}=0,\quad 
\gamma\in M'\}
$$
(see~(\ref{eq-trivial-dec})). 
\par
Now note that since $\relmul{M'}$ is free, the natural map
$\relmul{M'}\ra \relmul{M'}\otimes\Qq$ is injective.  Note also the
tautological embedding $\relmul{M}\fleche{i} \relmul{M'}$ induced by
the map $\Zz^M\ra \Zz^{M'}$ which maps any basis vector
$f_{\tilde{\alpha}}$ of $\Zz^M$ to $f_{\tilde{\alpha}^2}\in \Zz^{M'}$.
If $m\in \relmul{M}$, we have
$$
i(m)\in \{(n_{\gamma})\in
\Qq^{M'}\,\mid\, n_{\gamma}-n_{c(\gamma)}=0,\quad
\gamma\in M'\}
$$ 
and this means that $\relmul{M}=\reltriv{M}$, as desired.

\section{Application of the sieve for Frobenius}
\label{sec-sieve}

We are now going to apply the sieve for Frobenius to produce
extensions with Galois groups $W_{2g}^k$ to which we can apply the
results of Section~\ref{sec-algebraic} to prove Theorem~\ref{th-2} and
related results.
\par
For this we need to generalize the estimate for non-maximality of the
Galois group used in the proof of Proposition~\ref{pr-1} to situations
involving $W_{2g}^k$. For this purpose, we will again use sieve, and
we first recall the main statement for completeness. We use the
version from~\cite[Ch. 8]{lsieve} (the version in~\cite{k1} would also
suffice for our purposes), in the situation of a general
higher-dimensional parameter space. However, we extend it slightly to
allow tame ramification instead of prime-to-$p$ monodromy (see the
comments following the statement for a quick explanation if this is
unfamiliar).
\par
We will mention later on the (very small) improvements that can
sometimes be derived when the parameter space is a product of curves.

\begin{theorem}\label{th-sieve}
  Let $p$ be a prime number, $q\not=1$ a power of $p$. Let $V/\Fp_q$
  be a smooth affine geometrically connected algebraic variety of
  dimension $d\geq 1$. Assume $V$ can be embedded in $\mathbf{A}^N$
  using $r$ equations of degree $\leq \delta$, and assume also
  $\bar{V}$ has a compactification for which it is the complement of a
  divisor with normal crossing so that the tame (geometric)
  fundamental group $\pionet{V}$ is defined. Let $\Lambda$ be a set of
  primes $\ell\not=p$. For each $\ell\in\Lambda$, assume given a lisse
  sheaf $\sheaf{F}_{\ell}$ of $\Fp_{\ell}$-vector spaces,
  corresponding to an homomorphism
$$
\rho_{\ell}\,:\, \pione{V}\ra GL(r,\Fp_{\ell}),
$$
which is tamely ramified, so that $\rho_{\ell}$ restricted to the
geometric fundamental group factors through the tame quotient:
$$
\pioneg{V}\ra \pionet{V}\ra GL(r,\Fp_{\ell}).
$$
\par
Let $G_{\ell}$, $G^g_{\ell}$ be the corresponding arithmetic and
geometric monodromy groups, i.e.
$$
G_{\ell}=\rho_{\ell}(\pione{V}), \quad\quad
G^g_{\ell}=\rho_{\ell}(\pioneg{V})=\rho_{\ell}(\pionet{V}),
$$
and assume that for any distinct primes $\ell$, $\ell'\in\Lambda$, the
map
\begin{equation}\label{eq-disjoint}
\pioneg{V}\ra G_{\ell}^g\times G_{\ell'}^g
\end{equation}
is onto. 
\par
Let $\gamma_0\in G_{\ell}/G_{\ell}^g$ be the element such that all
the geometric conjugacy classes $\frob_t$ map to $\gamma_0$ for $t\in
V(\Fp_q)$, as in the short exact sequence
$$
1\ra G_{\ell}^g\ra G_{\ell}\ra G_{\ell}/G_{\ell}^g\ra 1,
$$
\par
Then for any choices of subsets $\Omega_{\ell}\subset G_{\ell}$ such
that the image of $\Omega_{\ell}$ in $G_{\ell}/G_{\ell}^g$ is
$\{\gamma_0\}$, and for any $L\geq 2$, we have
\begin{equation}\label{eq-large-sieve}
|\{t\in V(\Fp_q)\,\mid\,
\rho_{\ell}(\frob_t)\notin\Omega_{\ell}\text{ for all } \ell\leq L
\}|\leq 
(q^d+CL^Aq^{d-1/2})H^{-1}
\end{equation}
where, $\pi$ running over irreducible representations of $G_{\ell}$,
we have
\begin{gather}\label{eq-constants}
H=\sum_{\stacksum{\ell\leq L}{\ell\in\Lambda}}{
    \frac{|\Omega_{\ell}|}{|G^g_{\ell}|-|\Omega_{\ell}|}
  },\\
  A\leq 1+
\max_{\ell\leq L}
\Bigl\{
2\frac{\log |G_{\ell}|}{\log\ell}+
\max_{\pi}\frac{\log \dim \pi}{\log \ell}+
\sum_{\pi}{\frac{\log \dim_{\pi}}{\log \ell}}
\Bigr\}
\leq 
1+\frac{7}{2}\max_{\ell\leq L}{\frac{\log |G_{\ell}|}{\log\ell}},
\label{eq-constantA}
\\
C=12N2^r(3+r\delta)^{N+1}.
\label{eq-constantC}
\end{gather}
\end{theorem}

\begin{proof}
  The pieces are collected from~\cite[(8.11), Proposition
  8.7]{lsieve}, or the corresponding results in~\cite{k1} (where there
  is an extraneous factor $\kappa$ which can be removed as explained
  in~\cite{lsieve}). The only difference is the assumption that the
  sheaves are tamely ramified instead of the geometric monodromy
  groups being of order prime to $p$.  However the proof goes through
  with this weaker assumption, because the only place this was used
  was in applying the multiplicativity of the Euler-Poincar\'e
  characteristic in a finite Galois \'etale cover of degree prime to
  the characteristic. This result of Deligne and Lusztig holds for
  tamely ramified covers more generally (see~\cite[2.6, Corollaire
  2.8]{illusie}).
\end{proof}

\begin{remark}
  The generalization to tamely ramified sheaves is useful to avoid
  assuming that $p>2g+1$ when looking at families of curves to ensure
  that $Sp(2g,\Fp_{\ell})$ has order prime to $p$ (for instance, in
  Theorem~\ref{th-2}). The difference between the two is that tame
  ramification of an homomorphism $\pioneg{V}\ra
  GL(n,\Fp_{\ell})$ (with $\ell\not=p$) only requires that the
  $p$-Sylow subgroups of the ramification groups at infinity act
  trivially on $\Fp_{\ell}^n$, whereas having geometric monodromy
  group of order prime to $p$ means that the whole $p$-Sylow subgroup
  of the fundamental group acts trivially.
\par
Note however that in Remark~\ref{rm-kunneth}, we explain how one could
also prove Theorem~\ref{th-2} using only ramification theory for
curves.
\end{remark}

We derive from Theorem~\ref{th-sieve} a theorem generalizing the
maximality of splitting fields to $k\geq 2$. Recall first that a
family $(\sheaf{F}_{\ell})$ of lisse sheaves of free
$\Zz_{\ell}$-modules on an algebraic variety $V/\Fp_q$ is a compatible
system if, for any finite extension $\Fp_{q^{\nu}}/\Fp_q$, any $t\in
V(\Fp_{q^{\nu}})$, the characteristic polynomial
$$
\det(1-\frob_{q^{\nu},t}T\mid \sheaf{F}_{\ell})\in \Zz_{\ell}[T]
$$
is in fact in $\Zz[T]$ and is independent of $\ell$.

\begin{theorem}\label{th-many-k}
  Let $p$ be a prime number, $q\not=1$ a power of $p$, $g\geq 2$ and
  $k\geq 1$ integers. Let $V/\Fp_q$ be a smooth affine geometrically
  connected algebraic variety of dimension $d\geq 1$. Assume $V$ can
  be embedded in $\mathbf{A}^N$ using $r$ equations of degree $\leq
  \delta$, and define the constant $C(N,r,\delta)$ as
  in~(\ref{eq-constants}). Assume also $\bar{V}$ has a
  compactification for which it is the complement of a divisor with
  normal crossing so that the tame geometric fundamental group
  $\pionet{V}$ is defined.
\par
Let $\Lambda$ be a set of primes $\ell\not=p$ with positive
density, i.e., such that
\begin{equation}\label{eq-pos-dens}
\pi_{\Lambda}(L)=\sum_{\stacksum{\ell\leq L}{\ell\in\Lambda}}{1}
\gg \pi(L)
\end{equation}
for $L\geq L_0$, the smallest element of $\Lambda$, the implied
constant depending on $\Lambda$. For each $\ell\in\Lambda$, assume
given on $V$ a tamely ramified lisse sheaf $\tilde{\sheaf{F}}_{\ell}$
of free $\Zz_{\ell}$-modules of rank $2kg$ with $Sp(2g)^k$ symmetry,
i.e., given by representations
$$
\tilde{\rho}_{\ell}\,:\, \pione{V}\ra CSp(2g,\Zz_{\ell})^k.
$$
\par
Let $\tilde{\sheaf{F}}_{j,\ell}$ be the lisse sheaves given by
composition
$$
\pione{V}\ra CSp(2g,\Zz_{\ell})^k\ra CSp(2g,\Zz_{\ell}),
$$
and assume that for each $j$, $1\leq j\leq k$, the family
$(\tilde{\sheaf{F}}_{j,\ell})_{\ell\in\Lambda}$ is a compatible system.
\par
Then $(\tilde{\sheaf{F}}_{\ell})$ is also a compatible system; for
$t\in V(\Fp_q)$, let
$$
P_t=\det(1-\tilde{\rho}_{\ell}(\frob_t)T)\in\Zz[T].
$$
\par
Assume that this system has maximal geometric monodromy modulo $\ell$,
in the sense that the geometric monodromy group $G_{\ell}^g$ of
$\tilde{\sheaf{F}}_{\ell}/\ell\tilde{\sheaf{F}}_{\ell}$ is equal to
$G_{\ell}^g=Sp(2g,\Fp_{\ell})^k$ for all $\ell\in\Lambda$.
\par
Then we have
\begin{equation}\label{eq-goal-many}
|\{t\in V(\Fp_q)\,\mid\, \text{the splitting field of $P_t$ is not
  maximal} \}|\ll gc^kC^{2\gamma^{-1}}q^{d-\gamma^{-1}}(\log q)
\end{equation}
where $\gamma=29kg^2$, for some constant $c\geq 1$ depending only on
$g$, where the implied constant depends only on $\Lambda$. Here
maximality for $P_t$ means that the Galois group is isomorphic to
$W_{2g}^k$.
\par
Moreover, write $P_{j,t}=\det(1-T\frob_t\mid
\tilde{\sheaf{F}}_{j,\ell})$; then we also have
\begin{equation}\label{eq-trace-zero}
  |\{t\in V(\Fp_q)\,\mid\, \text{the sum of inverse roots of some
$P_{j,t}$ is zero}
  \}|\ll 
  kC^{2\gamma^{-1}}q^{d-\gamma^{-1}},
\end{equation}
where the implied constant depends only on $\Lambda$.
\end{theorem}

\begin{proof}
  First, notice that we have immediately the factorization
$$
\det(1-\tilde{\rho}_{\ell}(\frob_{q^{\nu},t})T)
=\prod_{1\leq j\leq k}{
\det(1-\tilde{\rho}_{j,\ell}(\frob_{q^{\nu},t})T)
}
$$
for any $t\in \Fp_{q^{\nu}}$, $\nu\geq 1$, so that the compatibility
of the systems $(\tilde{\sheaf{F}}_{j,\ell})_{\ell}$ implies that of
$(\tilde{\sheaf{F}}_{\ell})_{\ell}$, as stated. In particular, for
$t\in V(\Fp_q)$, we write
$$
P_t(T)=\prod_{1\leq j\leq k}{P_{j,t}(T)},\quad
\text{ with }\quad
P_{j,t}=\det(1-\tilde{\rho}_{j,\ell}(\frob_t)T).
$$
\par
Each $\tilde{\sheaf{F}}_{j,\ell}$ has maximal symplectic geometric
monodromy modulo $\ell$, since those monodromy groups are the images
of the composite
$$
\pioneg{V}\fleche{\rho_{\ell}} Sp(2g,\Fp_{\ell})^k
\ra Sp(2g,\Fp_{\ell})
$$
which are surjective (the first one by the maximal monodromy
assumption on $\sheaf{F}_{\ell}$). In particular, the splitting field
of $P_{j,t}$ over $\Qq$ has Galois group isomorphic to a subgroup of
$W_{2g}$ (by the customary functional equation), and the splitting
field of $P_{t}$ over $\Qq$ has Galois group isomorphic to a subgroup
of $W_{2g}^k$. This justifies the interpretation of the maximality
adjective in the statement of the theorem.
\par
We now recall the basic facts which allow sieve methods to detect
this type of maximality:
\par
-- For any $\ell\in\Lambda$, the reduction of $P_t$ modulo $\ell$ is
the characteristic polynomial of $\rho_{\ell}(\frob_t)$.
\par
-- If a polynomial $Q\in\Zz[T]$ of degree $r$ is such that $Q$ reduces
modulo a prime $\ell$ to a squarefree polynomial (of degree $r$) which
is the product of $n_1$ irreducible factors of degree $1$, \ldots,
$n_r$ irreducible factors of degree $r$, then as a subgroup of
permutations of the roots of $Q$, the Galois group of the splitting
field $Q$ contains an element with cycle structure consisting of $n_1$
fixed points, $n_2$ disjoint $2$-cycles, \ldots. 
\par
-- If a subgroup $H$ of a finite group $G$ has the property that
$H\cap c\not=\emptyset$ for all conjugacy classes $c\subset G$, then
$H=G$.
\par
Implementing this, let us first define a $q$-symplectic polynomial $R$
(with coefficient in a ring $B$) to be a polynomial in $B[T]$ of even
degree such that $R(0)=1$ and
$$
q^{(\deg P)/2}T^{(\deg P)}R\Bigl(\frac{1}{qT}\Bigr)=R(T),
$$
which is of course the ``functional equation'' for $\det(1-Tg)$ for
any symplectic similitude with multiplicator $q$. In particular, and
this is why we need the notion, the characteristic polynomials
$\det(1-\rho_{j,\ell}(\frob_t)T)$ are $q$-symplectic.
\par
In~\cite[Proof. of th. 8.13]{lsieve}, as in~\cite{k1}, we explicitly
described four subsets $\tilde{\Omega}_{1,\ell}$, \ldots,
$\tilde{\Omega}_{4,\ell}$ of $q$-symplectic polynomials of degree $2g$
in $\Fp_{\ell}[T]$ such that a $q$-symplectic polynomial in $\Zz[T]$
of degree $2g$ with non-maximal splitting field satisfies
$P\mods{\ell}\notin \tilde{\Omega}_{i,\ell}$ for some $i$ and all
$\ell$. From this, we construct the $4^k$ subsets
$$
\tilde{\Omega}_{\uple{i},\ell}
=
\prod_{1\leq j\leq k}{\tilde{\Omega}_{i_j,\ell}}
,\quad\quad
\uple{i}=(i_1,\ldots, i_k)\text{ with }
i_j\in \{1,2,3,4\},
$$
of the set of $q$-symplectic polynomials of degree $2kg$ in
$\Fp_{\ell}[T]$.
\par
It may be the case\footnote{\ We do not know if this happens or not.}
that a $P\in\Zz[T]$ which is $q$-symplectic of degree $2kg$ and splits
as
\begin{equation}\label{eq-fac}
P=P_1\cdots P_k,\quad\quad\text{$P_j\in\Zz[T]$, $q$-symplectic of
  degree $2g$}
\end{equation}
(so that the Galois group of its splitting field is a subgroup of
$W_{2g}^k$) has non-maximal splitting field but is not detected by
those subsets (i.e., for all $\uple{i}$, the factors $P_j$ reduce
modulo some $\ell$ to elements of $\tilde{\Omega}_{\uple{i},\ell}$):
the only obvious consequence here of the case $k=1$ is that the Galois
group of the splitting field, as a subgroup of $W_{2g}^k$, surjects to
each of the $k$-components $W_{2g}$.
\par
We bypass this problem by adding a fifth subset
$\tilde{\Omega}_{0,\ell}$ defined as
$$
\tilde{\Omega}_{0,\ell}=\{ f\in \Fp_{\ell}[T]\,\mid\, f\text{ is
  $q$-symplectic and is a product of $2g$ distinct linear factors} \}
$$
(which therefore corresponds to the trivial element of a Galois
group), and (re)define now $\tilde{\Omega}_{\uple{i},\ell}$ in the
obvious way for $\uple{i}$ a $k$-tuple with entries in
$\{0,1,2,3,4\}$. The point is that if a $q$-symplectic polynomial
$P\in\Zz[T]$ of degree $2kg$ factoring as above~(\ref{eq-fac}) has
splitting field strictly smaller than $W_{2g}^k$, then, for some
$\uple{i}\in \{0,1,2,3,4\}^k$, we have 
$$
(P_j\mods{\ell})_j\notin 
\tilde{\Omega}_{\uple{i},\ell}
$$ 
for all primes $\ell$. Indeed, arguing by contraposition, it would
follow otherwise by using 
$$
\uple{i}=(0,\ldots,0,i,0,\ldots, 0), \quad\quad 1\leq i\leq 4,
$$
(where the non-zero coordinate is the $j$-th one, $1\leq j\leq k$),
and the case $k=1$, that the Galois group, as a subgroup of
$W_{2g}^k$, contains 
$$
1\times \cdots \times 1 \times W_{2g}\times 1\cdots \times 1
$$
where the $W_{2g}$ occurs at the $j$-th position. Consequently, the
Galois group must be the whole of $W_{2g}$. In particular, we only
need to use the $4k$ tuples described in this argument.
\par
Now if we denote (with obvious notation for the multiplicator)
$$
\Omega_{\uple{i},\ell}=
\{g\in CSp(2g,\Fp_{\ell})^k,\ m(g)=(q,\ldots, q),\ 
\det(1-Tg)\in \tilde{\Omega}_{\uple{i},\ell}\}
$$
for $\ell\in\Lambda$, then we see that the left-hand side, say $N(L)$,
of~(\ref{eq-goal-many}) is at most
$$
N(L)\leq \sum_{\uple{i}}{
|\{t\in V(\Fp_q)\,\mid\, \det(1-T\rho_{\ell}(\frob_t))\notin
\Omega_{\uple{i},\ell},\ \text{ for $\ell\in\Lambda$}\}|}
$$
(where the sum ranges over the $4k$ tuples used before).
\par
Each of the terms in this sum may be estimated by the sieve for
Frobenius as in Theorem~\ref{th-sieve}, provided the last
assumption~(\ref{eq-disjoint}) is checked. Here it means showing that
$$
\pioneg{V}\ra Sp(2g,\Fp_{\ell})^k\times Sp(2g,\Fp_{\ell'})^k
$$
is onto, for $\ell\not=\ell'$ in $\Lambda$, and this follows from
Lemma~\ref{lm-goursat} below, which is a variant of Goursat's lemma. 
\par
The outcome of the sieve for Frobenius is the upper bound
$$
N(L)\leq 4k(q^d+CL^{A}q^{d-1/2})H^{-1}
$$
for $C$ given by~(\ref{eq-constantC}) and
$$
A\leq 29kg^2,\quad
\quad
H=\min_{\uple{i}} \sum_{\stacksum{\ell\leq L}{\ell\in\Lambda}}{
\frac{|\Omega_{\uple{i},\ell}|}{|Sp(2g,\Fp_{\ell})|^k}
}.
$$
\par
The former, which is quite rough but good enough for our purpose,
follows from the right-hand inequality in~(\ref{eq-constantA}),
together with the easy bound
$$
|CSp(2g,\Fp_{\ell})|\leq (\ell+1)^{2g^2+g+1},
$$
(note that the better bounds for the dimension and sum of dimension of
irreducible representations of $G_{\ell}$ which are described
in~\cite[Example 5.8 (2)]{lsieve} could also be used, if one tried to
optimize the value of $A$, e.g. for small values of $g$).
\par
To obtain a lower bound for $H$, we recall from~\cite[Proof. of
th. 8.13]{lsieve} again that there exists a constant $c_g>0$ (which
could also be specified more precisely) such that, for $\ell\geq 3$
and $1\leq i\leq 4$, we have
$$
\frac{|\tilde{\Omega}_{i,\ell}|}{|Sp(2g,\Fp_{\ell})|} \geq c_g,
$$
while the same counting arguments lead also to
$$
\frac{|\tilde{\Omega}_{0,\ell}|}{|Sp(2g,\Fp_{\ell})|} \geq c'_g,
$$
(see also~\cite[\S 3]{chavdarov},~\cite[App. B]{lsieve}) for $\ell\geq
2g+1$, for some other constant $c'_g$ (now extremely small, of the
order of $|W_{2g}|^{-1}$). Replacing $c_g$ by $\min(c_g,c'_g)$, we
have
$$
H\geq c_g^{-k} \pi_{\Lambda}(L)\gg c_g^{-k}\frac{L}{\log L},
$$
by~(\ref{eq-pos-dens}); this bound holds for $L>L_0$ and the implied
constant depending only on $\Lambda$ ($L_0$ can be taken as $\max
(2g+1,\text{smallest element of $\Lambda$}))$.
\par
The outcome is therefore that we have
$$
N(L)\ll 4kc_g^{-k}(q^d+CL^{A}q^{d-1/2})(\log L)L^{-1}.
$$
for $L>L_0$, the implied constant depending only on $\Lambda$.
\par
As usual, we select $L$ so that
$$
CL^A=q^{1/2},\quad\text{ i.e. }\quad
L=(qC^{-2})^{1/(2A)},
$$
if this is $>L_0$. This leads to
$$
N(L)\ll 4kc_g^{-k}q^{d-1/(2A)}(\log q)C^{1/A},
$$
where the implied constant depends only on $\Lambda$. This last
inequality is trivial if $L\leq L_0$ if we take the implied constant
large enough (indeed, if the implied constant is $\geq L_0\geq 2g+1$),
and so by doing so if necessary, we finish the proof
of~(\ref{eq-goal-many}).
\par
As for the proof of~(\ref{eq-trace-zero}), it follows the same idea,
but is much easier since we only need to ``sieve'' by a single
well-chosen prime $\ell\in\Lambda$ (what is called ``individual
equidistribution'' in~\cite{lsieve}, and is really the uniform
explicit Chebotarev Density Theorem here, as in~\cite{quad}). Indeed,
the sum of inverse roots of some $P_{j,t}$ is zero if and only if the
coefficient of $T$ in $P_{j,t}$ is zero.
\par
So, let $\tilde{\Upsilon}_{\ell}$ be the set of $q$-symplectic
polynomials of degree $2g$ in $\Fp_{\ell}[T]$ where the coefficient of
$T$ is non-zero, and $\Upsilon_{\ell}$ the set of matrices $g$ in
$CSp(2g,\Fp_{\ell})$ with multiplicator $q$ with
$\det(1-Tg)\in\Upsilon_{\ell}$. Then the left-hand side
of~(\ref{eq-trace-zero}) is bounded by
\begin{align*}
  M(\ell)&=|\{t\in V(\Fp_q)\,\mid\, P_{j,t}\mods{\ell}\notin
  \tilde{\Upsilon}_{\ell},\text{ for } 1\leq j\leq k
  \}|\\
  &=|\{t\in V(\Fp_q)\,\mid\, \rho_{j,\ell}(\frob_t)\notin
  \Upsilon_{\ell},\text{ for } 1\leq j\leq k \}|\\
&\leq |\{t\in V(\Fp_q)\,\mid\, \rho_{\ell}(\frob_t)\notin
  \Upsilon_{\ell}^k\}|,
\end{align*}
for any prime $\ell$. It is clear from the counting results
in~\cite[App. B]{lsieve} that we have
$$
\frac{|\Upsilon_{\ell}|}{
|Sp(2g,\Fp_{\ell})|}= 1+O(\ell^{-1}),\text{ and therefore }
\frac{|\Upsilon_{\ell}|^k}{
|Sp(2g,\Fp_{\ell})|^k}=1+O(k\ell^{-1})
$$
for all $\ell\geq 3$, $\ell\geq k$, the implied constant depending
only on $g$. Applying Theorem~\ref{th-sieve} with $\Lambda$ replaced
by $\{\ell\}$ for any fixed $\ell\in\Lambda$, we find
$$
M(\ell)\leq
(q^d+C\ell^Aq^{d-1/2})
\Bigl(1-\frac{|\Upsilon_{\ell}|^k}
{|Sp(2g,\Fp_{\ell})|^k}\Bigr)
\ll k(q^d+C\ell^Aq^{d-1/2})\ell^{-1},
$$
for $\ell\geq k$, the implied constant depending only on $g$, from
which the proof of~(\ref{eq-trace-zero}) finishes as before by
choosing a value of $\ell$ in a dyadic interval around the value
$(C^{-2}q)^{1/(2A)}$.
\end{proof}

Here is the group theoretic lemma we used in the proof.

\begin{lemma}
\label{lm-goursat}
Let $k\geq 1$ be an integer, $\ell_1$, $\ell_2$ distinct odd primes. Let
$G_1=Sp(2g,\Fp_{\ell_1})$ and $G_2=Sp(2g,\Fp_{\ell_2})$. If 
$H$ is a subgroup of $G_1^k\times G_2^k$ which surjects to $G_1^k$ and
to $G_2^k$ under the two projection maps, then in fact $H=G_1^k\times
G_2^k$.
\end{lemma}

\begin{proof}
  We can write $G_1^k\times G_2^k$ as a product of $2k$ factors, say
  $B_j$, $1\leq j\leq 2k$. Moreover, for any $i$, $j$, $1\leq i<j\leq
  2k$, the projection $H\ra B_i\times B_j$ is onto: this follows from
  the assumption if $B_i$ and $B_j$ are isomorphic (to $G_1$ or
  $G_2$), and from the usual Goursat lemma (as
  in~\cite[Prop. 5.1]{chavdarov}) if $B_i$ and $B_j$ are not.  Since
  moreover $G_1$ and $G_2$ are both equal to their commutator
  subgroups, the conclusion follows from~\cite[Lemma 5.2]{chavdarov}.
\end{proof}


\begin{remark}
  One can show that, for \emph{any} compatible system
  $(\sheaf{F}_{\ell})$ of lisse sheaves with $Sp(2g)^k$ monodromy, on
  a smooth curve over a finite field at least, there exists some
  compatible systems of lisse sheaves $(\sheaf{F}_{j,\ell})$, $1\leq
  j\leq k$, such that the monodromy of $\sheaf{F}_{j,\ell}$ is
  $Sp(2g)$ and the representation $\rho_{\ell}$ associated with
  $\sheaf{F}_{\ell}$ is given, up to isomorphism, by
\begin{equation}\label{eq-compat-split}
\rho_{\ell}(x)=(\rho_{j,\ell}(x))_{1\leq j\leq k},
\end{equation}
in terms of those associated with $\sheaf{F}_{j,\ell}$ (this amounts
to a choice of orderings of the projections
  $$
  p_j\,:\, Sp(2g)^k\ra Sp(2g),
  $$
  as $\ell$ varies, so that the sheaves $p_j(\sheaf{F}_{\ell})$ are
  compatible, for $1\leq j\leq k$).  This is a consequence of
  Lafforgue's proof of the global Langlands correspondance over
  function fields: fix some $\ell_0\not=p$, and define
  $\rho_{j,\ell_0}$ so that the formula above is valid for $\ell_0$;
  then Lafforgue shows there exists compatible systems
  $(\tilde{\rho}_{j,\ell})$ for which
  $\tilde{\rho}_{j,\ell}=\rho_{j,\ell_0}$ (see~\cite[Th. VII.6,
  (v)]{lafforgue}, using the fact that the geometric monodromy of
  $\tilde{\rho}_{j,\ell}$ is $Sp(2g)$, hence this sheaf is
  irreducible).  Define $\tilde{\rho}_{\ell}$ by the analogue
  of~(\ref{eq-compat-split}); then this compatible system (or its
  semisimplification) must be isomorphic to $\rho_{\ell}$ because they
  have same characteristic polynomials of Frobenius at all closed
  points.
\par
After twisting to reduce the $CSp(2g)^k$-case to $Sp(2g)^k$, this
means that the compatible systems considered in the theorem are very
likely the most general ones with the given monodromy for smooth
parameter spaces. It would be interesting to prove this directly and
in general, but this structure is obvious in our applications, so we
did not try to do this.
\end{remark}

\begin{remark}
  This theorem is interesting in itself as a complement to the earlier
  results of~\cite[\S 8]{lsieve} and~\cite{k1}: not only do most
  curves (in a family with large monodromy) have large Galois group,
  but their polynomial $L$-functions tend to be independent of each
  other. Note also that there are families of number fields fields
  which are pairwise linearly disjoint, but not globally disjoint (for
  instance, take $\Qq(\sqrt{2})$, $\Qq(\sqrt{3})$, $\Qq(\sqrt{6})$,
  where the compositum is biquadratic, and not of degree $8$),
  although if the Galois groups are perfect groups, pairwise
  disjointness does imply global disjointness (again by~\cite[Lemma
  5.2]{chavdarov}). Because $W_{2g}$ is not perfect,
  Theorem~\ref{th-many-k} can not be deduced directly from the cases
  $k=1$, $k=2$, and playing with intersections and
  inclusion/exclusion.
\end{remark}

\begin{corollary}\label{cor-indep}
  Let the data
  $(p,q,g,k,V/\Fp_q,N,r,\delta,d,\Lambda,(\tilde{\sheaf{F}}_{\ell}))$
  be as in Theorem~\ref{th-many-k} above. For $t\in V(\Fp_q)$, let
  $\zeros_t$ be the set of $\alpha$ such that
$$
\det(1-T\frob_{t}\mid\tilde{\sheaf{F}}_{\ell})
=\prod_{\alpha\in\zeros_t}{(1-\alpha T)},
$$
and let $\nzeros_t$ be the set of $\alpha /\sqrt{q}$ for
$\alpha\in\zeros_t$. Let $C$ be the constant defined
in~\emph{(\ref{eq-constantC})}. 
\par
Then we have
$$
|\{t\in V(\Fp_q)\,\mid\, \reladd{\zeros_t}\not=0\}|\ll
gc^kC^{2\gamma^{-1}}q^{d-\gamma^{-1}}(\log q),
$$
and
$$
|\{t\in V(\Fp_q)\,\mid\, \relmult{\nzeros_t}\not=0\}|\ll
gc^kC^{2\gamma^{-1}}q^{d-\gamma^{-1}}(\log q)
$$
for some constant $c\geq 1$ depending only on $g$, where
$\gamma=29gk^2$ and the implied constant depends only on $\Lambda$.
\end{corollary}

\begin{proof}
  As in the proof of Proposition~\ref{pr-1}, and with notation as in
  the statement of Theorem~\ref{th-many-k}, let us call \emph{special}
  any $t\in V(\Fp_q)$ such that:
\par
-- The splitting field of $\det(1-T\frob_t\mid
\tilde{\sheaf{F}}_{\ell})\in\Zz[T]$ (which is independent of $\ell$)
has Galois group $W_{2g}^k$.
\par
-- For some $j$, $1\leq j\leq k$,  the sum of the inverse roots of
$P_{j,t}$ is zero.
\par
By~(\ref{eq-goal-many}) and~(\ref{eq-trace-zero}), we have
$$
|\{t\in V(\Fp_q)\,\mid\, \text{$t$ is special}\}|
\ll gc^kC^{2\gamma^{-1}}q^{d-\gamma^{-1}}(\log q)
$$
where $\gamma=29kg^2$, for some constant $c\geq 1$ depending only on
$g$ where the implied constant depends only on $\Lambda$.
\par
Now, arguing exactly as in the proof of Proposition~\ref{pr-1} in
Section~\ref{sec-simple}, using Proposition~\ref{pr-w2g-k} (the first
part of which reduces the general case of arbitrary $k$ to that of
$k=1$ by excluding ``cross-relations''), we find that if $t$ is not
special, then there is no $\Qq$-linear dependency relation among the
$\alpha\in\zeros_t$, and also that the only multiplicative relations
among the $\tilde{\alpha}\in\nzeros_t$ are the obvious ones, which
concludes the proof.
\end{proof}

\section{Proof of Theorem~\ref{th-2}}
\label{sec-examples}

We can now prove Theorem~\ref{th-2} by direct applications of the
results of the previous section. First we state a lemma concerning
fundamental groups which seems to be well-known, but for which we
didn't find a reference in the literature. (It also holds in much
greater generality certainly, but we simply state what we need). The
argument of the proof was suggested by Q. Liu.

\begin{lemma}\label{lm-surject}
  Let $U$, $V$ be smooth affine connected schemes of finite type over
  the algebraic closure $k$ of a finite field. Fix a geometric point
  $\eta$ of $U\times V$, and let $\eta'$, $\eta''$ be its images in
  $U$ and $V$ respectively. Then the natural map
$$
\pi_1(U\times_k V,\eta)\fleche{\varphi} \pi_1(U,\eta')\times
\pi_1(V,\eta'')
$$
is surjective.
\end{lemma}

\begin{proof}
  We suppress the base points, which are fixed, for simplicity.  It
  suffices to show that the image $\Pi$ of the map is dense in
  $\pi_1(U)\times \pi_1(V)$, since $\Pi$ is closed ($\varphi$ is
  continuous and the fundamental groups are compact). This means that
  for any open set $W\subset \pi_1(U)\times \pi_1(V)$, we must show
  that $\Pi\cap W\not=\emptyset$. Since we have the product topology
  on the target, we may assume that $W=W_1\times W_2$, where
  $W_1\subset \pi_1(U)$, $W_2\subset \pi_1(V)$, are open. The
  profinite topology of the fundamental groups is also such that a
  basis of open sets are those of the form $W_i=x_i G_i$, where $x_i$
  is arbitrary and $G_i$ is a normal subgroup of finite index. Thus we
  must show that there exists $\sigma\in \Pi$ which is congruent to
  $x_1$ modulo $G_1$ and to $x_2$ modulo $G_2$, i.e.,
  $p_i(\sigma)=x_i\mods{G_i}$ where
$$
p_1\,:\, \pi_1(U)\ra \pi_1(U)/G_1=H_1,\quad\quad
p_2\,:\, \pi_1(V)\ra \pi_1(V)/G_2=H_2
$$
are the two projections.  If we let $E_1$ (resp. $E_2$) denote the
connected \'etale cover of $U$ (resp. $V$) associated with $G_1$
(resp. $G_2$), this means that we must find $\sigma\in H$ which acts
like $x_1$ on $E_1\ra U$ and like $x_2$ on $E_2\ra V$.
\par
However, let $E=E_1\times_k E_2$. Because $k$ is algebraically closed,
$E$ is a connected Galois covering of $U\times_k V$ with Galois group
$H_1\times H_2$, hence there is a surjective homomorphism
$$
\pi_1(U\times V)\ra H_1\times H_2,
$$
and $\sigma=\varphi(\sigma')$ will work for any $\sigma'\in
\pi_1(U\times V)$ which maps to $(x_1\mods{G_1},x_2\mods{G_2})$  under
this homomorphism.
\end{proof} 

\begin{remark}
  It is not the case that the map in Lemma~\ref{lm-surject} is
  injective in general. There are issues of wild ramification in
  positive characteristic which prevent this, see~\cite[Expos\'e X,
  Remarques 1.10]{sga1} for examples (even for $U=V$ the affine
  line). However, the prime-to-$p$ parts of $\pi_1(U\times V)$ and
  $\pi_1(U)\times \pi_1(V)$ are isomorphic (see~\cite[Expos\'e XIII,
  Proposition 4.6]{sga1}, under assumptions of existence of
  resolution of singularity, and~\cite{orgogozo} in general). More generally,
  the latter paper shows that there is isomorphism for the tame
  fundamental group (when this is defined).
\end{remark}

\begin{proof}[Proof of Theorem~\ref{th-2}]
  We will apply Theorem~\ref{th-many-k} with $V=U^k$, where $U$ is the
  complement of the set of zeros of the squarefree polynomial $f$
  defining the family of hyperelliptic curves.  The geometric
  parameters for $V$ are given by $N=2k$, $r=k$ and $\delta=2g+1$,
  since we can embed $U^k$ in $\mathbf{A}^{2k}$ (with coordinates
  $(x_j,y_j)$) using the $k$ equations
$$
x_jf(y_j)=1,\quad\quad 1\leq j\leq k.
$$
\par
Thus the constant $C$ in~(\ref{eq-constantC}) satisfies
$$
C\leq 24(2g+1)2^k(3+(2g+2)k)^{2k+1}
$$
(notice this constant grows superexponentially in terms of $k$, but it
will be raised to a very small power later on; going back to the
original proof of the large sieve inequality in this particular case,
one can replace this constant by one which grows ``only''
exponentially, see Remark~\ref{rm-kunneth}; the improvements on the
final results are barely visible).
\par
Since $\bar{U}$ is the complement of $2g+1$ points in the projective
line $\mathbf{P}^1/\bar{\Fp}_q$, $\bar{V}$ is the complement of
$(2g+1)^k$ coordinate hyperplanes in $\mathbf{P}^k/\bar{\Fp}_q$,
which form a divisor with normal crossings, so that the tame
fundamental group is well-defined for $\bar{V}$.
\par
Let $f\,:\, \mathcal{C}\ra U$ be the morphism defining the
(compactified) family of curves, which we recall are given by the
affine equations
$$
C_t\,:\, y^2=f(x)(x-t),
$$
and let
$$
p_j\,:\, V\ra U,\quad\quad 1\leq j\leq k,
$$
denote the coordinate projections. We use the family of sheaves
$$
\tilde{\sheaf{F}}_{\ell}=\bigoplus_{1\leq j\leq k}{
p_j^* R^1f_!\Zz_{\ell},
}
$$
for $\ell\in\Lambda$, the set of odd primes $\not=p$. By construction,
the associated sheaves $\tilde{\sheaf{F}}_{j,\ell}$ are each copies of
$R^1f_!\Zz_{\ell}$, and hence they form compatible systems of lisse
sheaves of free $\Zz_{\ell}$-modules of rank $2g$, in fact with
$$
\det(1-T\frob_{q^{\nu},t}\mid R^1f_!\Zz_{\ell})=
P_{C_t}(T)\in\Zz[T],\quad\quad \text{for $\nu\geq 1$, $t\in
  U(\Fp_{q^{\nu}})$}
$$
\par
Each $R^1f_!\Zz_{\ell}$, for $\ell\geq 3$, $\ell\not=p$,
corresponds to a homomorphism 
$$
\rho'_{\ell}\,:\, \pione{U}\ra CSp(2g,\Zz_{\ell}),
$$
which is tamely ramified (see~\cite[Lemma 10.1.12]{katz-sarnak}), the
symplectic structure coming from Poincar\'e duality for curves. In
turn, $\sheaf{F}_{\ell}$ is also tamely ramified. Indeed, the
corresponding homomorphism, restricted to the geometric fundamental
group, factors as follows:
$$
\pioneg{V}\ra \pioneg{U}^k \ra \pionet{U}^k\ra Sp(2g,\Zz_{\ell})^k,
$$
and it is essentially tautological\footnote{\ This amounts to saying
  that $V\fleche{p_j} U$ induces an homomorphism on the respective
  tame fundamental groups.} that for all $j$, the $j$-th component
homomorphism
$$
\pioneg{V}\ra \pionet{U}
$$
also factors through $\pionet{V}$; consequently, the original
homomorphism also factors through the tame fundamental group of
$\bar{V}$.
\par
From all this, it follows that
$(\tilde{\sheaf{F}}_{\ell})_{\ell\in\Lambda}$, is a compatible system
of free $\Zz_{\ell}$-modules of rank $2kg$, which is tamely ramified,
and such that we have
$$
\det(1-T\frob_{\uple{t}}\mid \tilde{\sheaf{F}}_{\ell})=
\prod_{1\leq j\leq k}{P_{C_{t_j}}(T)}
$$
for any $\uple{t}=(t_1,\ldots, t_k)\in V(\Fp_q)$.
\par
To compute the geometric monodromy group of
$\sheaf{F}_{\ell}=\tilde{\sheaf{F}}_{\ell}/\ell\tilde{\sheaf{F}}_{\ell}$,
we appeal to Lemma~\ref{lm-surject} and induction to ensure that we
have a surjective homomorphism
\begin{equation}\label{eq-surjective}
\pioneg{V}\fleche{\prod p_{j,*}}\pioneg{U}^k
\end{equation}
and we observe that the representation $\rho_{\ell}$ corresponding to
$\sheaf{F}_{\ell}$ factors as
\begin{equation}\label{eq-surj2}
  \pione{V}\fleche{\prod p_{j,*}}
  \pione{U}^k\ra CSp(2g,\Fp_{\ell})^k,
\end{equation}
the last homomorphism being $(\rho'_{\ell},\ldots, \rho'_{\ell})$
where $\rho'_{\ell}$ corresponds to the sheaf $R^1f_!\Fp_{\ell}$ on
$U$. 
\par
Then we invoke (as in~\cite{chavdarov},~\cite{k1},~\cite{lsieve} for
$k=1$) the remarkable theorem of J-K. Yu according to which the image
of $\rho'_{\ell}$ restricted to $\pioneg{U}$ (i.e., the geometric
monodromy group modulo $\ell$) is equal to $Sp(2g,\Fp_{\ell})$ for all
odd primes (C. Hall~\cite{hall} has given another proof, whereas Yu's
proof is unpublished).  This together with~(\ref{eq-surj2})
and~(\ref{eq-surjective}) immediately implies that the geometric
monodromy group $G_{\ell}^g$ of $\sheaf{F}_{\ell}$ is
$Sp(2g,\Fp_{\ell})^k$, as needed to apply Theorem~\ref{th-many-k}.
\par
We note also that the value of $C$ above, and $\gamma=29kg^2$, leads
by trivial bounds to
$$
C^{2\gamma^{-1}}\ll k^{(4g^2)^{-1}},
$$
for $g\geq 1$, $k\geq 1$, with an absolute implied constant.  Applying
Corollary~\ref{cor-indep}, we find that the number of $\uple{t}\in
V(\Fp_q)$ for which either
$\reladd{\zeros(\uple{C}_{\uple{t}})}\not=0$ or
$\relmult{\nzeros(\uple{C}_{\uple{t}})}\not=0$ is at most
$$
\ll gc^k k^{(4g^2)^{-1}}q^{k-\gamma^{-1}}(\log q) \ll c_1^k
q^{k-\gamma^{-1}}(\log q)
$$
for any $c_1>c$, where the implied constants depends only on $g$. This
concludes the proof of Theorem~\ref{th-2}.
\end{proof}

It is clear that, \emph{mutatis mutandis}, we have proved the
following more general statement
instead of Theorem~\ref{th-2}:

\begin{proposition}
  Let $p$ be a prime number, $q\not=1$ a power of $p$, and $k\geq 1$
  integers. Let $U_1$, \ldots, $U_k$ be smooth affine curves over
  $\Fp_q$ and
$$
\mathcal{C}_j\fleche{f_j} U_j
$$
families of smooth projective curves of genus $g_j\geq 1$ such that,
for some set $\Lambda$ of primes of positive density, the geometric
monodromy of $R^1f_{j,!}\Fp_{\ell}$ is $Sp(2g_j,\Fp_{\ell})$ for
$\ell\in\Lambda$.  Let $U=U_1\times\cdots\times U_k$.
\par
Then, with obvious notation, we have
\begin{gather*}
  |\{\uple{t}\in U(\Fp_q)\,\mid\,
  \reladd{\zeros(\uple{C}_{\uple{t}})}\not=0 \}|\ll
  c^kq^{k-\gamma^{-1}}(\log q),
  \\
  |\{\uple{t}\in U(\Fp_q)\,\mid\,
  \relmult{\nzeros(\uple{C}_{\uple{t}})}\not=0 \}|\ll
  c^kq^{k-\gamma^{-1}}(\log q),
\end{gather*}
where $\gamma=29(g_1^2+\cdots+g_k^2)>0$ for some constant $c\geq 1$
depending only on $(g_1,\ldots, g_k)$.  In both estimates, the implied
constant depends only on $\Lambda$, $(g_1,\ldots, g_k)$ and the
Euler-Poincar\'e characteristic of the curves $\bar{U}_i$.
\end{proposition}


\begin{remark}\label{rm-kunneth}
  We explain now how to replace the constant $C$
  in~(\ref{eq-large-sieve}) by a smaller one in the case above where
  $V=U^k$ with $U$ a smooth affine curve, complement of the zeros of a
  polynomial $f$ of degree $m$ in the affine line.
\par
More precisely, in Theorem~\ref{th-sieve}, suppose that $V$ is of this
type. Let $p_i$, $1\leq i\leq k$, denote the $i$-th coordinate map
$V\ra U$. Assume then that we have sheaves $(\sheaf{G}_{\ell})$ on the
curve $U$ which arise by reduction modulo $\ell$ from a compatible
system $(\tilde{\sheaf{G}}_{\ell})_{\ell}$ such that the sheaves
$\sheaf{F}_{\ell}$ are given by
$$
\sheaf{F}_{\ell}=\bigoplus_{1\leq j\leq k}{p_j^*\sheaf{G}_{\ell}}
$$
(note it is not necessary here to assume that the sheaves are tamely
ramified, but they must form a compatible system, which is not assumed
in~Theorem~\ref{th-sieve}). Let $\rho_{\ell}$ (resp. $\tau_{\ell}$) be
the representations of $\pione{V}$ associated to $\sheaf{F}_{\ell}$
(resp. $\sheaf{G}_{\ell}$).
\par
From the proof of the large sieve inequality and the setting of the
sieve for Frobenius, a bound for $C$ derives from a uniform estimate
for the ``exponential sums''
$$
S(\pi,\pi')=\sum_{\uple{t}\in V(\Fp_q)^k}{
\Tr(\pi(\rho_{\ell}(\frob_{\uple{t}})))
\overline{
\Tr(\pi'(\rho_{\ell'}(\frob_{\uple{t}})))}
}
$$
for primes $\ell$, $\ell'\in \Lambda$ and irreducible representations
$\pi$ (resp. $\pi'$) of $G_{\ell}$ (resp. $G_{\ell'}$); see~\cite[\S
2.2; Prop. 2.9; \S 8.3]{lsieve}.
\par
For sheaves of the type above, the monodromy group $G_{\ell}$ of
$\sheaf{F}_{\ell}$ is clearly isomorphic to $H_{\ell}^k$, where
$H_{\ell}$ is the monodromy group of
$\sheaf{G}_{\ell}$. Correspondingly, the representations $\pi$ and
$\pi'$ factor as external tensor products
$$
\pi=\boxtimes_{1\leq j\leq k} \pi_j,\quad\quad 
\pi'=\boxtimes_{1\leq j\leq k} \pi'_j,
$$
where $\pi_j$ (resp. $\pi'_j$) are uniquely-defined irreducible
representations of $G_{\ell}$ (resp. $G_{\ell'}$), and since
$\frob_{\uple{t}}=(\frob_{t_1},\ldots,\frob_{t_k})$, the exponential
sum itself factors\footnote{\ Cohomologically speaking, this reflects
  the K\"unneth formula for the groups
  $H^i_c(\bar{V},\pi(\sheaf{F}_{\ell})\otimes
  \pi'(\sheaf{F}_{\ell'}))$ (where the tensor product is the external
  one if $\ell\not=\ell'$) which occur after applying the
  Grothendieck-Lucite's trace formula directly to $S(\pi,\pi')$.} as
$$
S(\pi,\pi')=
\prod_{1\leq j\leq k}{\ 
\sum_{t\in U(\Fp_q)}{
\Tr(\pi_j(\tau_{\ell}(\frob_{t}})))
\overline{
\Tr(\pi'(\tau_{\ell'}(\frob_{t})))}
},
$$
where each term is now a $1$-variable sum of the type discussed for
the large sieve on a parameter curve. Using the bounds
in~\cite[Prop. 8.6 (2), Prop. 8.7]{lsieve}, it is easy to deduce that
the constant $C$ in~(\ref{eq-constants}) may be replaced with
$$
C'=(1-\chi_c(\bar{U})+mw)^k,
$$
where $w$ is the sum of Swan conductors of $\sheaf{F}$ at the points
at infinity (see~\cite[\S 4]{k1} for the definition; it vanishes in
the case of tame ramification). Thus we obtain a bound which ``only''
grows exponentially in $k$. However, this turns out to be a fairly
inconsequential gain in the applications in this paper at least.
\end{remark}

\section{Examples of relations among zeros}
\label{sec-numerics}

In this section we wish to give explicit examples of $L$-functions
over finite fields where the (inverse) roots satisfy some
multiplicative relations (for additive relations, see
Remark~\ref{rm-frob-tori-lindep}). Numerically, we tried to find such
relations by looking (using \textsc{GP}'s function \texttt{lindep})
for ``small'' dependency relations between the components of the
vectors $(\pi,\theta_1,\ldots, \theta_g)$, where $\pm \theta_j\in
[0,2\pi[$ are the arguments of the $2g$ inverse roots considered. It
is easy to confirm rigorously a relation obtained this way, since all
numbers involved are algebraic (but on the other hand, if some of the
large relations found by \texttt{lindep} are genuine, we have missed
them...)
\par
It is interesting to remark here that in the case of linear relations
between roots of unrestricted rational polynomials, Berry, Dubickas,
Elkies, Poonen and Smyth~\cite{bdeps} have found for any integer
$n\geq 1$ what is the largest degree $d=d(n)$ for which there exists
an algebraic number $\alpha$ of degree $d$ over $\Qq$ such that its
conjugates span a $\Qq$-vector space of dimension $n$; in fact, they
show that $d(n)$ is the same as the maximal order of a finite subgroup
of $GL(n,\Qq)$, and then invoke results of Feit, Weisfeiler -- which
depend on the classification of finite simple groups -- that give this
value. As we already recalled at the beginning of
Section~\ref{sec-algebraic}, except for seven exceptional cases, such
a group is isomorphic to $W_{2n}$, so that $d(n)=2^n n!$. Among the
remaining cases, for instance, we have $d(4)=1152$. There are also
similar (less complete) results for multiplicative relations.

\begin{example}
  We started by looking at purely numerical examples using previous
  computations of roughly $160000$ zeta functions of hyperelliptic
  curves of genus $3$ in two particular families of the type occurring
  in Theorem~\ref{th-2} (computed using \textsc{Magma}~\cite{magma},
  see~\cite[End of \S 8.6]{lsieve}), over fields $\Fp_{5^k}$, $k\leq
  8$. We had found only about $50$ non-irreducible $L$-functions,
  and among these only three curves over $\Fp_{5^8}$ in the family
$$
y^2=(x^2+6x-1)(x-t)
$$
which have irreducible polynomial $L$-functions (of degree $6$) having
Galois groups the dihedral group $D_{12}$. However, upon examination
of the roots, it turns out that there are no non-trivial relations
(although there certainly exist self-reciprocal polynomials with this
Galois group and some interesting multiplicative relations).
\par
This confirms of course the ``genericity'' of the independence of the
roots, and suggests that the upper bounds in Proposition~\ref{pr-1}
are far from the truth (however, we only did very spotty checks for
relations involving multiple zeta functions, i.e., corresponding to
$k\geq 2$).
\end{example}

\begin{example}
  In view of the lack of success of the previous item, a natural way
  to try to construct examples without looking at curves directly is
  to use the fact that for (most) choice of polynomial $P$ satisfying
  the functional equation (with respect to a power of prime $q\not=1$)
  and Riemann Hypothesis, there exists, if not an algebraic curve $C$,
  at least an abelian variety $A/\Fp_q$ where the $L$-function (more
  precisely, the reversed characteristic polynomial of the geometric
  Frobenius acting on $H^1(\bar{A},\Zz_{\ell})$, which we call the
  $L$-function to simplify) is exactly given by this polynomial.  This
  is due to Honda and Tate (see~\cite{tate}) and allows us to simply
  look for polynomials with roots satisfying non-trivial relations.
\par
One simple way to do this is to consider $q=p$ and take a polynomial
which splits as a product
$$
\prod_{1\leq j\leq g}{(1-a_jT+pT^2)}
$$
where $a_j\in\Zz-\{0\}$ (to avoid ordinarity issues) satisfies
$|a_j|<2\sqrt{p}$. Honda-Tate theory then implies that this polynomial
\emph{is} the $L$-function for some abelian variety $A/\Fp_p$ of
dimension $g$, which is in fact isogenous to the product of the
elliptic curves corresponding to the factors $1-a_jT+pT^2$. Since the
inverse roots $\alpha_j$, $\beta_j$ with
$$
\prod_{1\leq j\leq g}{(1-a_jT+pT^2)}=
\prod_{1\leq j\leq g}{(1-\alpha_jT)(1-\beta_jT)}
$$
are given by
$$
\alpha_j=\frac{a_j+i\sqrt{4p-a_j^2}}{2},\quad\quad
\beta_j=\frac{a_j-i\sqrt{4p-a_j^2}}{2},
$$
one can try to select $p$ and $a_j$ so that the quadratic fields
$\Qq(i\sqrt{4p-a_j^2})$ are identical for all $j$; this locates all
$2g$ roots in the same imaginary quadratic field, and one may hope for
non-trivial relations. Of course we can take $a_j=a$ for all $j$, but
this is cheating, and similarly using signs $a_j=\pm a$ leads to
factors which are all geometrically isomorphic elliptic curves. More
interestingly, one should look for $a_j$'s with distinct absolute
values, so that $A$ becomes a product of $g$ pairwise non-isogenous
elliptic curves.
\par
This can happen, but this type of behavior is actually pretty
restricted: we need to find distinct $a_j$'s, and integers $f_j$, such
that
$$
4p=a_j^2+df_j^2,\quad 1\leq j\leq g,
$$
for a common squarefree value of $d$. This means that
$$
p=N_{\Qq(\sqrt{-d})/\Qq}\Bigl(\frac{a_j}{2}+\frac{f_j\sqrt{-d}}{2}
\Bigr),
$$
and by standard properties of quadratic fields, the ideal
$\mathfrak{a}$ generated by $w_j=\frac{a_j}{2}+\frac{f_j\sqrt{-d}}{2}$
in the ring of integers of $\Qq(\sqrt{-d})$ is unique up to
conjugation.\footnote{\ This $w_j$ is necessarily an integer because
  its norm ($p$) and its trace ($a_j$) are.} The only way to obtain
distinct values is therefore to replace $w_j$ by some other generator
of $\mathfrak{a}$, i.e., by $\eps w_j$ where $\eps\in \Qq(\sqrt{-d})$
is a unit.  If $\Qq(\sqrt{-d})$ is of discriminant $\not=-4$, $-3$,
only $-w_j$ is permitted, which simply amounts to replacing $a_j$ by
$-a_j$.  So the interesting possibilities are when $d=1$ or $d=3$.
\par
In the first case, the units are $\pm 1$, $\pm i$, and if we write
$p=N_{\Qq(i)/\Qq}(a/2+ib/2)$, then besides $a_1=|a|$, we can take
$a_2=|b|$ to obtain the two distinct positive solutions.  Note
moreover that this is possible if and only if $p\equiv 1\mods{4}$ by
Fermat's theorem on primes which are sums of two squares.
\par
In the second case where $d=3$, which can occur if and only if $4p$ is
of the form $a^2+3b^2$, i.e., if and only if $p\equiv 1\mods{3}$,
there are six units, equal to $\pm 1$, $\pm j$, $\pm j^2$ where
$j=(-1+i\sqrt{3})/2$. Writing
$$
p=N_{\Qq(\sqrt{-3})/\Qq}\Bigl(\frac{a}{2}+\frac{b\sqrt{-3}}{2}\Bigr),
$$
with $a\geq 1$, $b\geq 1$ integers, and multiplying by $j$ and $j^2$,
we find that there are three possible (positive) values for $a$,
namely
$$
a,\quad \frac{a+3b}{2},\quad \frac{|a-3b|}{2}.
$$
\par
Note in passing the following amusing property: if those three values
(say $x$, $y$, $z$) are ordered so that $x<y<z$, then we have
$z=x+y$. Indeed, this amounts to the identities
\begin{gather*}
\frac{a+3b}{2}+\frac{a-3b}{2}=a,\quad\text{ if }
a>3b,\\
\frac{3b-a}{2}+a=\frac{a+3b}{2},\quad\text{ if }
a<3b.
\end{gather*}
\par
Here is a simple example for $d=3$, with $g=3$, $p=541$ (the $100$-th
prime); we find that the three values of $a$ are $a_1=17$, $a_2=29$,
$a_3=46$, and indeed we have
$$
4p-a_1^2=1875=3\cdot 5^4,\quad
4p-a_2^2=1323=3^3\cdot 7^2,\quad
4p-a_3^2=48=3\cdot 2^4,
$$
so the corresponding inverse roots are 
$$
\alpha_1=\frac{17+25i\sqrt{3}}{2},\quad
\alpha_2=\frac{29+21i\sqrt{3}}{2},\quad
\alpha_3=\frac{46+4i\sqrt{3}}{2},
$$
in $\Qq(\sqrt{-3})$. If we let $\tilde{\alpha}_j=\alpha_j/\sqrt{p}$,
then the reader will easily check that we have the relation
$$
\tilde{\alpha}_1^2\tilde{\alpha}_2^{-4}\tilde{\alpha}_3^2=1.
$$
\end{example}

\begin{example}
  Another type of examples can be obtained from the work of
  Katz~\cite{katz-g2} on $G_2$-equidistribution for some families of
  exponential sums. Precisely, for $p\not=2$, $7$, consider the
  exponential sums defined by
$$
S_m(t)=\sum_{x\in\Fp_{q^m}^{\times}}{ \chi_2(N_{\Fp_{q^m}/\Fp_q})(x))
  e\Bigl(\frac{\Tr_{\Fp_{q^m}/\Fp_p}(x^7+tx)}{p}\Bigr)},\quad\quad
m\geq 1,\quad t\in \Fp_{q},\quad q=p^{\nu},
$$
where $\chi_2$ is the quadratic character of $\Fp_q$. Katz shows that
it has the property that, for $t\in\Fp_q$, the zeta function
$$
\exp\Bigl(\sum_{m\geq 1}{S_m(t)\frac{T^m}{m}}\Bigr)
$$
is a polynomial of degree $7$ in $\Zz[\zeta_7][T]$, where $\zeta_7$ is
a primitive $m$-th root of unity, and that when properly normalized by
dividing by $(-G)^m$, where $G$ is the Gauss sum given by
$$
G=\sum_{x\in
  \Fp_q^{\times}}{\chi_2(x)e\Bigl(-7\frac{\Tr_{\Fp_q/\Fp_p}(x)}{p}\Bigr)},
$$
it is the characteristic polynomial of a semisimple matrix in
$SO(7,\Cc)$ which lies in a conjugate of the exceptional group
$G_2$. By the known structure of a maximal torus in such a group (as
explained in~\cite[(5.5)]{katz-g2}), its inverse roots are of the form
\begin{equation}\label{eq-g2-torus}
(1,\tilde{\alpha},\tilde{\beta},\tilde{\alpha}\tilde{\beta}
,\tilde{\alpha}^{-1},\tilde{\beta}^{-1},(\tilde{\alpha}\tilde{\beta})^{-1}),
\end{equation}
and we see clearly some interesting relations.
\par
Performing the computations (with \textsc{Magma}) for $p=5$, $t=1$, we
obtain that the inverse roots are $\sqrt{5}$, and numbers given
approximately by
\begin{gather*}
\alpha=1.809016994374947424102293417 - i  \cdot
1.314327780297834015064172712,
\quad\quad
5/\alpha=\bar{\alpha},\\
\beta=-1.225699835949638884074294475 + i \cdot
1.870203174030305277157650105,
\quad\quad
5/\beta=\bar{\beta},\\
\gamma=0.1076658471997440358697076407 - i \cdot
2.233474438032985720105383483,
\quad\quad
5/\gamma=\bar{\gamma},
\end{gather*}
the first two of which are roots of 
$$
P_1=X^4 - 5X^3 + 15X^2 - 25X +25
$$ 
(which is Galois over $\Qq$ with cyclic group $\Zz/4\Zz$), while the
other four are roots of
$$
P_2=X^8 + 5X^6 - 20X^5 + 5X^4 - 100X^3 + 125X^2 + 625,
$$
which has (non-abelian) splitting field of degree $16$ over $\Qq$ (the
Galois group is generated by the permutations 
$(1\ 2\ 6\ 5)(3\ 7\ 4\ 8)$ and $(2\ 4)(3\ 5)$, for some
ordering of the roots). Corresponding to the
pattern~(\ref{eq-g2-torus}), one finds that
$$
\frac{\alpha}{\sqrt{5}}\cdot \frac{\beta}{\sqrt{5}}\cdot
\frac{\gamma}{\sqrt{5}}=1.
$$
\par
Note that one finds that there are four roots (not related by
inversion), say $x$, $y$, $z$, $t$, of $P_2$ which satisfy a relation
$x^{-1} y^3 z t^{-3} = 1$.  So it would be interesting to know if this
polynomial $P_2$ corresponds to an algebraic curve of genus $4$ over
$\Fp_5$ (experimentally, what would be the number of points of this
curve over $\Fp_{5^n}$, i.e.,
$$
5^n+1-(x^n+y^n+z^n+t^n+x^{-n}+y^{-n}+z^{-n}+t^{-n})
$$
are non-negative integers, as they should; the sequence starts $6$,
$36$, $66$, $596$, $3126$,..., and only the first two terms could have
been negative).
\end{example}

\begin{example}
  Other systematic investigations can be done in cases where the zeta
  functions of families of curves are explicitly known, or computable
  with easily available tools. We first discuss briefly some examples
  related to modular curves (see the next example for the case of
  Fermat curves).
\par
Let $N\geq 1$ be an integer, and consider the modular curve $X_0(N)$
over the finite field $\Fp_p$ for some $p\nmid N$. From
Eichler-Shimura theory and Atkin-Lehner theory, the polynomial
$L$-function of $X_0(N)/\Fp_p$ is given by
$$
P_{N}(T)=\prod_{f}{(1-a_f(p)T+pT^2)^{m(f)}}
$$
where $f$ runs over the finite set of primitive forms of weight $2$
for any $\Gamma_0(M)$ where $M\mid N$ (``newforms'' in Atkin-Lehner
terminology), with $a_f(p)$ being the $p$-th Hecke eigenvalue of
$f$. If $f$ is of conductor $M$, then the multiplicity $m(f)$ of $f$
is $d(N/M)$, the number of divisors of $N/M$. This is often $\geq 2$,
showing the existence of multiple roots of $P_N$, hence of some
multiplicative relations. It is natural to restrict to the ``new''
part, which means taking instead of $X_0(N)$ the new part
$J_0(N)^{new}$ of its Jacobian variety. The $L$-function of this
abelian variety is
$$
P_{N}^*(T)=\prod_{\text{$f$ level $N$}}{(1-a_f(p)T+pT^2)}\in \Zz[T]
$$
(note that $P_N^*=P_N$ if $N$ is a prime for instance). 
\par
The $a_f(p)$ are totally real algebraic integers, with usually
distinct degrees. We used \textsc{Magma} to compute some of the
polynomials $P_N^*$, taking levels $N$ prime roughly up to $300$ and
primes $p$ in $\{5,7,11,13\}$ (coprime with $N$); this amounts to
about $1000$ cases.
\par
What happens experimentally is that a large majority (roughly $85\%$)
of the splitting fields of $1-a_f(p)T+pT^2$ (over $\Qq$) have Galois
group $W_{2\deg(a_f(p))}$. This does not exclude cross-relations for
different $f$ of the same level, but small-scale tests only found a
few of those in remaining multiple factors (e.g., $(1+T+5T^2)^2$
divides $P^*_{167}$ for the prime $p=5$).
\par
Even when the Galois group of a factor is smaller than
$W_{2\deg(a_f(p))}$, most of the time there is no extra relation. The
few exceptions correspond to factors of degree $4$ of the type
$$
1-aT^2+p^2T^{4}
$$
(e.g., $1+17T^2+121T^4$ divides $P_{67}^*$ and $P_{313}^*$ for the
prime $11$, $1+6T^2+49T^4$ divides $P_{29}^*$ for the prime $7$),
where there are relations of the type $\alpha^2=\beta^2$. Similar even
polynomials could probably occur also for other values.
\end{example}

\begin{example}
  Let $F_{m}$ be the Fermat curve defined by
$$
F_m\,:\, x^m+y^m+z^m=0
$$
in the projective plane (more general diagonal hypersurfaces could
also be considered). The zeta functions of these curves over all
finite fields are well-known, going back to Weil at least. We assume
$q\not=1$ is a power of a prime for which $q\equiv 1 \mods{m}$ and we
consider $F_m/\Fp_q$. Let then $X_m$ be the set of $m-1$ non-trivial
characters in the cyclic group (of order $m$) of characters of order
$m$ of $\Fp_q^{\times}$. Let
$$
g(\chi)=\sum_{x\in \Fp_q^{\times}}{\chi(x)e(\Tr(x)/p)}
$$
for $\chi\in X_m$ be the associated Gauss sum, and let moreover $A_m$
be the set of $3$-tuples $(\chi_0,\chi_1,\chi_2)\in X_m^3$ such that
$\chi_0\chi_1\chi_2$ is trivial.  Then (see, e.g.,~\cite[\S 11.3,
Th. 2]{ireland-rosen}), the $L$-function of $F_m$ is the
polynomial
$$
\prod_{(\chi_0,\chi_1,\chi_2)\in A_m}{(
1-q^{-1}g(\chi_0)g(\chi_1)g(\chi_2)T
)}\in \Zz[T],
$$
so that, in particular, the distinct normalized inverse roots are the
numbers
\begin{equation}\label{eq-inv-gauss}
\frac{g(\chi_0)g(\chi_1)g(\chi_2)}{q^{3/2}},\quad\quad
(\chi_0,\chi_1,\chi_2)\in A_m.
\end{equation}
\par
We can see here many multiplicative relations: first of all, in $A_m$,
permutations are permitted, and since the inverse roots only depend on
the set $\{\chi_0,\chi_1,\chi_2\}$, there will typically be
multiplicities among the numbers~(\ref{eq-inv-gauss}), which is of
course a well-known fact.\footnote{\ It is interesting to note that
  Ulmer~\cite{ulmer} has recently used properties of zeta functions of
  Fermat curves to construct examples of abelian varieties
  $A/\Fp_q(t)$ which have bounded ranks in towers of extensions of the
  form $\bar{\Fp}_q(t^{1/d})$, $d$ ranging over powers of suitable
  primes or integers not divisible by $p$; the crucial properties for
  him are however the prime factorizations of the inverse roots.}
\par
But even among roots taken without the obvious multiplicities arising
from permutations, and with only one of each pair $(\alpha,q/\alpha)$
preserved, non-trivial relations will arise because the order of $A_m$
modulo those restrictions grows quicker than $m$. Indeed, let $B_m$ be
the set of different triplets $(\chi_0,\chi_1,\chi_2)$ modulo
permutations, and modulo the inversions. Since we have
$$
|A_m|=\frac{(m-1)^3-(m-1)}{m},
$$
there are at least $|A_m|/12$ elements in $B_m$, which is roughly
$m^2/12$ as $m$ gets large. A product restricted to
representatives of $B_m$, with exponents $\uple{n}=(n_b)$, leads to an
expression of the type
$$
q^{-3 U(\uple{n})/2}\prod_{\chi\in X_m}{g(\chi)^{u_{\chi}(\uple{n})}}
$$
where the $u_{\chi}$ are linear forms with integral coefficients and
$U(\uple{n})$ is the sum of the $n_b$. So, to produce a relation, it
\emph{suffices} to find $\uple{n}$ such that
$$
U(\uple{n})=0\quad\text{ and }\quad 
u_{\chi}(\uple{n})=0,\quad\text{ for } \chi\in X_m.
$$
\par
These are $m$ linear relations with integral coefficients, so quite
quickly there will less of them than there are coefficients available,
guaranteeing the existence of non-zero solutions.
\par
Here is the example of $m=7$: denoting the characters in $X_m$ by
$\omega^j$, $1\leq j\leq 6$, for some generator $\omega$, there are
$30$ elements in $A_7$, and $8$ basic triplets up to permutation
(listed as exponents of $\omega$), namely
$$
(1,1,5)\quad (1,2,4)\quad (1,3,3)\quad (2,2,3)\quad (2,6,6)\quad
(3,5,6)\quad (4,4,6)\quad (4,5,5);
$$
\par
Among those, it is easy to check that the inverse
roots~(\ref{eq-inv-gauss}) corresponding to the last four ones are
inverses of those corresponding to the first four ones, leaving $4$
elements in $B_7$. Then one finds that the matrix of equations, with
columns indexed by the remaining $4$ triplets in order, is
$$
\begin{pmatrix}
1 & 1 & 1 & 1 \\
2 & 1 & 1 & 0 \\
-1 & 1 & 0 & 2\\
0 & -1 & 2 & 1 \\
0 & 1 & -2 & -1\\
1 & -1 & 0 & -2\\
-2 & -1 & -1 &0
\end{pmatrix}
$$
and even though we still have more relations than parameters in this
particular case, one checks that the integral kernel of this matrix is
non-zero, being of rank $1$ and generated by the row vector
$$
(1, -1, -1, 1)
$$
(in fact the first equation $U(\uple{n})=0$ is redundant here, since
the sum of coefficients in each column is constant). This means for
any Fermat curve $F_7$ over $\Fp_q$ with $q\equiv 1\mods{7}$, there
will be four roots $\tilde{\alpha}_1$, \ldots, $\tilde{\alpha}_4$,
such that
$$
\tilde{\alpha}_1\tilde{\alpha}_2^{-1}
\tilde{\alpha}_3^{-1}\tilde{\alpha}_4=1. 
$$
\end{example}

\section{Frobenius tori and multiplicative independence}
\label{sec-frob-tori}

In this section we review briefly the theory of Frobenius tori of
Serre, in the version of C. Chin~\cite[\S 5]{chin}, and explain how it
leads to more direct proofs of statements of multiplicative
independence of normalized zeros of $L$-functions in the case of
families with large symplectic monodromy. In
Remark~\ref{rm-frob-tori-lindep}, we give examples showing that, on
the other hand, this technique does not lead (at least directly) to
results concernant linear independence.
\par
Consider a finite field $\Fp_q$ and a continuous representation
$$
\Gal(\bar{\Fp}_q/\Fp_q)\fleche{\rho} GL(r,\Qq_{\ell})
$$
for some $\ell\not=p$.  Serre defines the Frobenius torus $\Tt_{\rho}$
associated to $\rho$ to be the connected component of the identity of
the diagonalizable algebraic group $\Hh_{\rho}/\Qq_{\ell}$ which is
the Zariski closure in $GL(r)/\Qq_{\ell}$ of the subgroup generated by
the semisimple part of $\rho(\frob_{\Fp_q})$. The character group
$\Hom(\Hh_{\rho},\G_m)$ of $\Hh_{\rho}$ is canonically isomorphic to
the multiplicative group $\mult{M_{\rho}}$ generated by the set
$$
M_{\rho}=\{\lambda_1,\ldots,\lambda_r\}
$$
of eigenvalues of $\rho(\frob_{\Fp_q})$. Since a diagonalizable group
is determined by its character group (see, e.g.,~\cite[\S
3.2]{springer} for the basic theory), and since there is an exact
sequence
$$
0\ra \relmul{M_{\rho}}\ra  \Zz^r\ra\mult{M_{\rho}}\ra 0,
$$
we see that to know the group $H_{\rho}$ is equivalent to knowing the
group of multiplicative relations among the eigenvalues of
$\rho(\frob_{\Fp_q})$, showing the relevance of this theory to
questions of multiplicative independence of Frobenius eigenvalues.
\par
It follows, in particular, that if the image of $\rho$ lies in a group
(isomorphic to) $Sp(2g,\Qq_{\ell})^k$ for some non-degenerate
alternating pairing and $k\geq 1$, then we have:
$$
\text{$\relmult{M_{\rho}}=0$ if and only if $\Tt_{\rho}=\Hh_{\rho}$ is
  a maximal torus in $Sp(2g)^k/{\Qq_{\ell}}$.}
$$
\par
Serre proved the first statement of the following type in the case of
abelian varieties over number fields; the precise statement is a very
special case of~\cite[Th. 5.7]{chin}:

\begin{theorem}[J-P. Serre; C. Chin]\label{th-chin}
  Let $V/\Fp_q$ be a smooth affine algebraic variety of dimension
  $d\geq 1$. Let $k\geq 1$, and let
$$
\rho\,:\, \pione{V}\ra Sp(2g,\Qq_{\ell})^k
$$
be a continuous representation such that the image of
$\pioneg{V}$ under $\rho$ is Zariski dense in the
algebraic group $Sp(2g)^k/\Qq_{\ell}$.  Assume that the following
conditions hold:
\par
\emph{(1)} The representation $\rho$ is pointwise pure of weight $0$;
\par
\emph{(2)} There exists $C\geq 0$ such that, for every closed point
$x$ of $V$, with residue field of degree $n\geq 1$ over $\Fp_q$, every
eigenvalue $\alpha$ of $\rho(\frob_x)$ and every $p$-adic valuation
$v$ of $\Qq(\alpha)$, we have
$$
|v(\alpha)|\leq C |v(q^n)|.
$$
\par
\emph{(3)} There exists $D\geq 0$ such that, for every closed point
$x$ of $V$, with residue field of degree $n\geq 1$ over $\Fp_q$, every
eigenvalue $\alpha$ of $\rho(\frob_x)$ and every $p$-adic valuation
$v$ of $\Qq(\alpha)$, we have
$$
D\frac{v(\alpha)}{v(q^n)}\in \Zz.
$$
\par
Then there exists a non-empty conjugacy-invariant Zariski open subset
$W_k\subset Sp(2g)^k/\Qq_{\ell}$ such that, for any $x\in V(\Fp_q)$, the
Frobenius torus $\Tt_x$ associated to the local representation
$\rho_x$ defined by the composite
$$
\spec(\Fp_q)\fleche{x} \pione{V} \fleche{\rho}
Sp(2g,\Qq_{\ell})^k
$$
is a maximal torus in $Sp(2g)^k/\Qq_{\ell}$ if $\rho(\frob_{x,\Fp_q})\in
W_k$.
\end{theorem}

Consider now the situation of Theorem~\ref{th-2}, for a fixed value of
$k\geq 1$: $f\in\Zz[X]$ is a squarefree monic polynomial of degree
$2g$, where $g\geq 1$ is an integer, $p$ is an odd prime such that $p$
does not divide the discriminant of $f$. Let $U/\Fp_p$ be the open
subset of the affine line where $f(t)\not=0$, and denote again by
$\mathcal{C}_f\ra U$ the family of hyperelliptic curves defined in
Proposition~\ref{pr-1}. Fix an odd prime $\ell\not=p$ such that $q$ is
a square in $\Qq_{\ell}$, and consider the lisse $\Qq_{\ell}$-sheaf
$\rho$
corresponding to
$$
\bigoplus_{1\leq j\leq k}{R^1p_{j,!}\Qq_{\ell}}.
$$
\par
Fixing a square-root $\alpha=\sqrt{p}\in\Qq_{\ell}$, we can form the
rank $1$ sheaf $\alpha^{-\deg(\cdot)}=\Qq_{\ell}(1/2)$ on $U^k$ (see
the discussion in~\cite[9.1.9]{katz-sarnak}), and
the twist
$$
\Bigl(\bigoplus_{1\leq j\leq k}{R^1p_{j,!}\Qq_{\ell}}\Bigr)(1/2)
$$
which has the property that the corresponding representation
$\rho'=\rho\otimes \alpha^{\deg(\cdot)}$ takes value in the group
$Sp(2g,\Qq_{\ell})^k$ (instead of $CSp(2g,\Qq_{\ell})^k$), and which
is pointwise pure of weight $0$ by the Riemann Hypothesis for curves
over finite fields.  Other well-known properties of curves over finite
fields imply that conditions (2) and (3) of Theorem~\ref{th-chin} hold
for $\rho'$. The last condition, in particular, has to do with
$p$-adic divisibility properties of the zeros of the $L$-functions of
the curves in the family, and can be obtained for instance from
Honda-Tate theory (see, e.g.,~\cite{tate}).\footnote{\ For more
  complicated sheaves, checking this assumption typically involves
  crystalline cohomology; see, e.g,~\cite[Th. 3.2]{chin} where it is
  proved to hold, using the techniques of Lafforgue's proof of the
  global Langlands correspondance over function fields, for any lisse
  sheaf which is irreducible with determinant of finite order (e.g.,
  trivial) on a smooth curve.}
\par
Thus, we deduce from Theorem~\ref{th-chin} and the remark before the
statement of this theorem that there exists a non-empty
conjugacy-invariant Zariski dense subset $W_k\subset Sp(2g)^k$ such
that, for any power $q=p^n\not=1$, and for $\uple{t}\in U(\Fp_q)^k$,
we have
$$
\relmult{\nzeros(\uple{C}_{\uple{t}})}=0
$$
\emph{unless} $\rho'(\frob_{\uple{t},\Fp_q})\in W_k$. Hence, defining
$C_k$ to be the closed complement of $W_k$ in $Sp(2g)^k$, we have
$$
 |\{\uple{t}\in U(\Fp_q)^k\,\mid\,
  \relmult{\nzeros(\uple{C}_{\uple{t}})}\not=0 \}|
\leq 
 |\{\uple{t}\in U(\Fp_q)^k\,\mid\, \rho'(\frob_{\uple{t},\Fp_q})\in
 C_k\}|.
$$
\par
Since $C_k$ is closed of dimension $<\dim Sp(2g)^k$, we can apply
Deligne's Equidistribution Theorem to deduce directly
$$
\lim_{q\ra +\infty}{
\frac{1}{q^k}|\{\uple{t}\in U(\Fp_q)^k\,\mid\,
  \relmult{\nzeros(\uple{C}_{\uple{t}})}\not=0 \}|
}=0.
$$
\par
This is a qualitative statement, but it can be made quantitative, for
fixed $k$, by appealing to an explicit uniform Chebotarev density
theorem, as in~\cite{quad}, and by reduction modulo $\ell^m$ for some
well-chosen $m\geq 1$. Precisely, $\rho'$ has a natural
$\Zz_{\ell}$-structure, and by the monodromy result of J-K. Yu
(already used in the proof of Theorem~\ref{th-2}) and some fairly
standard group theory, the homomorphisms
$$
\pi_1(\bar{U}^k,\bar{\eta})\ra Sp(2g,\Zz/\ell^m \Zz)^k
$$
are surjective for all $m\geq 1$. Since $C_k$ is a proper closed
subset of $Sp(2g)^k$, the order of the image $C_{k,m}$ of $C_k$ modulo
$\ell^m$ satisfies
\begin{equation}\label{eq-codim}
\frac{|C_{k,m}|}{|Sp(2g,\Zz/\ell^m\Zz)^k|}\ll \frac{1}{\ell^m}
\end{equation}
for $m\geq 1$, where the implied constant depends only on $k$. We have
$$
 |\{\uple{t}\in U(\Fp_q)^k\,\mid\, \rho'(\frob_{\uple{t},\Fp_q})\in
 C_k\}|
\leq
|\{\uple{t}\in U(\Fp_q)^k\,\mid\, \rho'(\frob_{\uple{t},\Fp_q})\in
 C_{k,m}\}|.
$$
\par
By the Chebotarev density theorem  we obtain
\begin{multline*}
|\{\uple{t}\in U(\Fp_q)^k\,\mid\, \rho'(\frob_{\uple{t},\Fp_q})\in
 C_{k,m}\}|=\frac{|C_{k,m}|}{|Sp(2g,\Zz/\ell^m\Zz)^k|} q^k\\
+O(q^{k-1/2}
|Sp(2g,\Zz/\ell^m\Zz)^k||C_{k,m}|^{1/2}),
\end{multline*}
where the implied constant depends only on $g$ and $k$ (as
in~\cite[Th. 1.1]{quad}, but arguing as in the beginning of
Theorem~\ref{th-2} to estimate the relevant sum of Betti numbers in
such a way that the dependency only involves $g$ and
$k$). Using~(\ref{eq-codim}) and rough estimates, this gives
$$
|\{\uple{t}\in U(\Fp_q)^k\,\mid\, \rho'(\frob_{\uple{t},\Fp_q})\in
 C_{k,m}\}|
\ll \ell^{-m}q^k+O(q^{k-1/2}\ell^{6mg^2k}),
$$
and the choice of $m$ as the integer $m\geq 1$ such that
$$
\frac{1}{2(6g^2k+1)}\frac{\log q}{\log \ell}-1\leq m<
\frac{1}{2(6g^2k+1)}\frac{\log q}{\log \ell}
$$
(if $m$ exists, but otherwise the result becomes trivial) leads to
$$
 |\{\uple{t}\in U(\Fp_q)^k\,\mid\, \rho'(\frob_{\uple{t},\Fp_q})\in
 C_k\}|
\ll
\ell q^{k-\gamma^{-1}}
$$
with $\gamma=2(6g^2k+1)$, where the implied constant depends only on
$g$ and on $k$. 
\par
Compared to Theorem~\ref{th-2}, two issues arise. The first is the
apparent dependency on the choice of $\ell$ such that $q$ is a square
in $\Qq_{\ell}$, but this is mostly cosmetic. It is clear that one can
take $\ell\leq p$, so the ``vertical'' direction $q=p^n$ with
$n\ra+\infty$ is dealt in this manner. As is, the ``horizontal''
direction $q=p\ra +\infty$ requires something close to the Generalized
Riemann Hypothesis (over $\Qq$) (which implies $\ell\ll (\log p)^2$),
but it is also certainly possible to prove directly a variant of
Theorem~\ref{th-chin} for sheaves of weight $1$ to avoid the twist by
$\alpha^{\deg(\cdot)}$ required to obtain a sheaf of weight $0$.
\par
The second issue is the uniformity in terms of $k$, which is more
delicate, but would be necessary to obtain a result as strong as
Theorem~\ref{th-2}. To deal with it, one needs to write down more
explicitly the closed subset $C_k$ that occurs in the proof, and more
precisely (by the standard point-counting estimates for varieties over
finite fields), one needs to have an estimate for the number of
(geometric) irreducible components of $C_k$. Since $C_k$ is described
quite concretely in~\cite[p.37]{chin}, obtaining a bound seems
feasible ($C_k$ is the union of Zariski closures of conjugates of a
finite set of subgroups $\Hh$ of a maximal torus such that the
connected component of $\Hh$ is among a finite set of subtori, and has
index $\leq N$, where $N$ is some integer depending only on $g$), but
it isn't obvious (to the author) how to do it efficiently. Certainly,
counting only the number of subgroups $\Hh$ leads to a bound worse
than exponential in terms of $k$, which would give a worse dependency
on $k$ than Theorem~\ref{th-2}, but one may hope that not all
subgroups lead to different irreducible components after conjugation
and taking the Zariski closure.

\begin{remark}
  Another interesting contrast between this proof of
  Theorem~\ref{th-2} (for fixed $k$) and the previous one is that this
  one depends crucially on using  $p$-adic information
  about the eigenvalues (via Condition (3) of Theorem~\ref{th-chin}),
  whereas the first one doesn't require any $p$-adic input (if one
  uses Remark~\ref{rm-kunneth}, at least, because otherwise there is,
  hidden in the proof of the necessary estimates for sums of Betti
  numbers, some $p$-adic arguments of Bombieri and
  Adolphson-Sperber). 
\end{remark}

\begin{remark}\label{rm-frob-tori-lindep}
  We now show that the Frobenius torus does not control the linear
  relations between the Frobenius eigenvalues. Let $A=E_1\times
  \cdots\times E_g$ be the product of $g$ pairwise non-isogenous
  ordinary elliptic curves over a finite field $\Fp_q$. Then, for a
  fixed prime $\ell\not=p$, the Frobenius torus of $A$ (which
  corresponds to $\Tt_{\rho}$ for the representation $\rho$ on
  $H^1(A,\Qq_{\ell})$, twisted as before so that the eigenvalues are
  of modulus $1$) is a maximal torus of $Sp(2g)/{\Qq_{\ell}}$.
\par
Let $(\lambda_i,q/\lambda_i)$ denote the eigenvalues of the Frobenius
automorphism for $E_i$, and let $a_i=\lambda_i+q/\lambda_i$ be the
trace of Frobenius, which is an integer. The set
$$
M_A=\{\lambda_1,q/\lambda_1,\ldots, \lambda_g, q/\lambda_g\},
$$
is the set of all Frobenius eigenvalues of $A$.  So we see that any
non-trivial linear relation between the $a_i$'s (which exist in
abundance) gives a non-trivial linear relation between the elements of
$M_A$: defining $T_A=\{a_1,\ldots, a_g\}$, there is an injection
$$
\left\{
\begin{matrix}
\reladd{T_A}&\injecte &\reladd{M_A}&
\\
(n_i)&\mapsto & (m_{\lambda}),&\text{ where }
m_{\lambda_i}=m_{q/\lambda_i}=n_i,
\end{matrix}
\right.,
$$
and since $\reladd{T_A}$ is a $\Zz$-module of rank $g-1$ (the $a_i$
being non-zero), the rank of $\reladd{M_A}$ is $\geq g-1$.
\par
(Note that, conversely, we can fix arbitrarily a $g$-tuple
$(n_i)_{1\leq i\leq g}$, and then find, for all prime powers $q=p^k$
large enough, some $(a_i)_{1\leq i\leq g}$ with $p\nmid a_i$ and
$|a_i|\leq 2\sqrt{q}$, such that
$$
\sum_{1\leq i\leq g}{n_ia_i}=0,
$$
and building the corresponding elliptic curves, this shows that any
arbitrarily fixed linear relation of this type can be obtained from an
abelian variety with maximal Frobenius torus.)
\end{remark}

\end{document}